\documentclass[preprint,12pt]{elsarticle}

\usepackage{amsfonts}
\usepackage{amssymb}
\usepackage{float}
\usepackage{amsmath}
\usepackage{amsthm}
\usepackage{physics}
\usepackage{xcolor}
\usepackage{makeidx}

\usepackage{lineno}
\usepackage{hyperref}
\usepackage[noabbrev]{cleveref}

\def\Rpp{\R_{++}}

\def\Sc{\mathbb{S}}
\def\Sn{\Sc^n}

\def\Snp{\Sc_+^n}

\def\Snpp{\Sc_{++}^n}

\def\R{\mathbb{R}}
\def\Rd{\mathbb{R}^d}

\def\Rdn{\R^{d\times n}}
\def\Rnd{\R^{n\times d}}
\def\Rmn{\R^{m\times n}}
\def\Rd{\mathbb{R}^d}
\def\Rnod{\R^{(n-1)\times d}}

\def\Rn{\mathbb{R}^n}

\DeclareMathOperator{\argmin}{{argmin}}
\DeclareMathOperator{\diag}{{diag}}
\DeclareMathOperator{\Diag}{{Diag}}

\newcommand{\EDM}{\textbf{EDM}\,}
\newcommand{\EDMp}{\textbf{EDM}}
\newcommand{\lngm}{\textbf{lngm}\,}
\newcommand{\lngmp}{\textbf{lngm}}

\newcommand{\KK}{{\mathcal K} }
\newcommand{\cK}{{\mathcal K} }
\newcommand{\cS}{{\mathcal S} }
\newcommand{\textdef}[1]{\textit{#1}\index{#1}}

\newcommand{\cM}{{\mathcal M}}
\newcommand{\cD}{{\mathcal D} }
\newcommand{\SC}{\mathcal{S}^n_C}
\newcommand{\SH}{{\mathcal S}^n_H}

\newcommand{\cO}{{\mathcal O} }
\newcommand{\Rdd}{\R^{d \times d}}
\newcommand{\Rnn}{\R^{n \times n}}
\newcommand{\cL}{{\mathcal L} }
\newcommand{\cT}{{\mathcal T} }

\newcommand{\tred}[1]{\textcolor{black}{#1}}

\numberwithin{figure}{section}
\numberwithin{equation}{section}

\newtheorem{theorem}{Theorem}[section]
\newtheorem{proposition}[theorem]{Proposition}

\newtheorem{example}[theorem]{Example}

\newtheorem{cor}[theorem]{Corollary}
\newtheorem{corollary}[theorem]{Corollary}

\newtheorem{remark}[theorem]{Remark}

\newtheorem{problem}[theorem]{Problem}

\newtheorem{lemma}[theorem]{Lemma}

\crefname{theorem}{Theorem}{Theorems}
\Crefname{theorem}{Theorem}{Theorems}
\crefname{thm}{Theorem}{Theorems}
\Crefname{thm}{Theorem}{Theorems}
\crefname{problem}{Problem}{Problems}
\Crefname{problem}{Problem}{Problems}
\Crefname{assump}{Assumption}{Assumptions}
\crefname{assump}{Assumption}{Assumptions}
\crefname{conjecture}{Conjecture}{Conjectures}
\Crefname{conjecture}{Conjecture}{Conjectures}
\crefname{proposition}{Proposition}{Propositions}
\Crefname{proposition}{Proposition}{Propositions}
\crefname{prop}{Proposition}{Propositions}
\Crefname{prop}{Proposition}{Propositions}
\crefname{cor}{Corollary}{Corollaries}
\Crefname{cor}{Corollary}{Corollaries}
\crefname{lem}{Lemma}{Lemmas}
\Crefname{lem}{Lemma}{Lemmas}
\crefname{defn}{definition}{Definitions}
\Crefname{defn}{Definition}{Definitions}
\crefname{conj}{Conjecture}{Conjectures}
\Crefname{conj}{Conjecture}{Conjectures}
\crefname{remark}{Remark}{Remarks}
\Crefname{remark}{Remark}{Remarks}
\crefname{rmk}{Remark}{Remarks}
\Crefname{rmk}{Remark}{Remarks}
\crefname{example}{Example}{Examples}
\Crefname{example}{Example}{Examples}
\Crefname{case}{Case}{Cases}
\Crefname{Case}{Case}{Cases}
\crefname{align}{}{}
\Crefname{align}{}{}
\crefname{equation}{}{}
\Crefname{equation}{}{}

\DeclareMathOperator{\nul}{null}
\DeclareMathOperator{\range}{range}

\DeclareMathOperator{\kvec}{{vec}}
\DeclareMathOperator{\kMat}{{Mat}}

\DeclareMathOperator{\Mat}{{Mat}}

\DeclareMathOperator{\LTriag}{{\cL Triag}}

\DeclareMathOperator{\svec}{{svec}}

\DeclareMathOperator{\spanl}{{span}}

\makeindex

\journal{Linear Algebra and its Applications}

\begin{document}

\begin{frontmatter}



\title{On the local and global minimizers of the smooth stress function
in Euclidean distance matrix problems\footnote{Date: Aug. 13, 2024;
Revision: July 19, 2025}}



\author[af0,af1,af2]{Mengmeng Song} \fnref{fn0}
\fntext[fn0]{The work of this author was partially supported by the Academic Excellence Foundation of BUAA for PhD Students, and international joint doctoral education fund of Beihang University.}
            
\author[af3]{Douglas S. Gon\c{c}alves\fnref{fn1}\corref{cr1}}
\fntext[fn1]{The work of this author was supported in part by CNPq (Grant 305213/2021-0) and FAPESC (Grant 2024TR002238).}

\cortext[cr1]{Corresponding author.}
\ead{douglas.goncalves@ufsc.br}
            
\author[af2]{Woosuk L. Jung}

\author[af4]{Carlile Lavor\fnref{fn2}}
\fntext[fn2]{The work of this author is supported in part by FAPESP (Grant 2023/08706-1) and CNPq (Grants 305227/2022-0 and 404616/2024-0).}
            
\author[af5]{Antonio Mucherino\fnref{fn3}}
\fntext[fn3]{The work of this author was partially supported by the ANR project EVARISTE (ANR-24-CE23-1621).}
            
\author[af2]{Henry Wolkowicz\fnref{fn4}}
\fntext[fn4]{Research supported by The Natural Sciences and Engineering Research Council of Canada.}

\address[af0]{National Frontiers Science Center for Industrial Intelligence and Systems Optimization, Northeastern University, Shenyang, P.R. China}
\address[af1]{School of Mathematical Sciences, Beihang University, Beijing, 100191, P.R. China}
\address[af2]{Department of Combinatorics and Optimization, Faculty of Mathematics, University of Waterloo, N2L 3G1, Waterloo, Canada}
\address[af3]{Department of Mathematics, Universidade Federal de Santa Catarina, 88040-900, Florianop\'{o}lis, Brazil}
\address[af4]{Department of Applied Mathematics (IMECC), University of Campinas, 13083-859, Campinas, Brazil}
\address[af5]{IRISA, Universit\'{e} de Rennes, F-35042, Rennes, France}
            
\begin{abstract}
We consider the nonconvex minimization problem, with quartic objective function, that arises in the exact recovery of a configuration matrix
$P\in \R^{nd}$ of $n$ points when a Euclidean distance matrix, \EDMp, is given with embedding dimension $d$. 
It is an open question in the literature whether there are
conditions such that the
minimization problem admits a local nonglobal minimizer, \lngmp.
We prove that all second-order stationary points are global minimizers whenever $n \leq d + 1$. 
{And, for $d=1$ and $n\geq 7>d+1$,
we present an example where we can analytically exhibit a local nonglobal minimizer. 
For more general cases,} we numerically find a second-order stationary point and then prove that there indeed exists a nearby
\lngm for the quartic nonconvex minimization problem.
Thus, we answer the previously open question about
their existence in the affirmative.
Our approach to finding the \lngm is novel in that we first exploit the
translation and rotation invariance to remove the singularities of
the Hessian, and reduce the size of the problem
from $nd$ variables in $P$ to $(n-1)d - d(d-1)/2$ variables.
This allows for stabilizing Newton's method, and for
finding examples that satisfy the strict second 
order sufficient optimality conditions.

The motivation for being able to find global minima is to
obtain \emph{exact recovery} of the configuration matrix,
even in the cases where the data is noisy and/or incomplete,
without resorting to approximating solutions
from convex (semidefinite programming) relaxations.
In the process of our work we present new insights into when \lngmp s 
of the smooth stress function do and do not exist.
\end{abstract}

%


\end{frontmatter}


\tableofcontents
\label{toc:toc}
\listoffigures

\section{Introduction}

\EDM (completion) problems have long been studied in
the scientific literature, see e.g.,~the
surveys, book collections, and some recent papers~\cite{KrislockWolk:10,MR3246296,MR3059590,alfm18,all19,bsl20}.
It is well known that one can obtain the Gram matrix $\bar G$ from a given \EDM $\bar D$.
Then, a configuration matrix $\bar P$ of
points $\bar p_i \in\Rd$ such that $\bar D_{ij}=\|\bar p_i-\bar p_j\|^2,
i,j = 1,\ldots,n$, can be obtained from a full rank factorization 
\index{configuration matrix, $\bar P$}
\index{$\bar P$, configuration matrix}
\[
\bar G = \bar P\bar P^T, \quad \bar P^T = 
          \begin{bmatrix}\bar p_1,\ldots, \bar p_n  \end{bmatrix}\in
               \R^{d\times n}.
\]

In this paper, we consider the question of exact recovery from 
the unconstrained minimization problem
\begin{equation}\label{eq:s2P}
\min_{P \in \Rnd} \quad \| \KK(PP^T) - \bar{D} \|_F^2,
\end{equation}
where $\KK: \Sn
\to \Sn$ is the Lindenstrauss operator on symmetric matrix space:
\index{$\diag(G)$, diagonal of $G$}
\index{diagonal of $G$, $\diag(G)$}
\index{$e$, vector of ones}
\index{vector of ones, $e$}
\label{page:vones}
\[
   \KK(G) = \diag(G)e^T+e\diag(G)^T-2G. 
\]
Here $e$ is the vector of ones and $\diag(G)$ is the linear mapping
providing the vector of diagonal elements of the square matrix $G$. 
{Moreover, $\|\cdot \|$, $\|\cdot\|_F$,
denote the Euclidean and Frobenius norms, respectively.}
\label{page:2}

The optimization problem~\cref{eq:s2P} is a
Euclidean distance geometry problem (DGP), see e.g.,~\cite{ MR3246296}. 
DGP includes the \EDM Completion Problem,
where $\bar D$ can have missing entries as well as noisy entries.
This latter problem has been proven to be NP-Hard~\cite{saxe}.

The objective function of \cref{eq:s2P} is denoted and expressed as a quartic in $P$:
\index{smooth stress, $\sigma_2(P)$} 
\index{$\sigma_2(P)$, smooth stress} 
\label{page:quartichard}
$$\sigma_2(P) = \| \KK(PP^T) - \bar{D} \|_F^2 = \sum_{i=1}^n \sum_{j=1}^n \left( \| p_i - p_j \|^2 - \| \bar{p}_i - \bar{p}_j \|^2 \right)^2,
$$
which is referred to as the \textdef{smooth stress} 
in the multidimensional scaling (MDS) literature. 
{(Throughout this text, for notational convenience, we use
\textdef{$f(P)=\frac 12 \sigma_2(P)$} as our objective function.)} 
\index{$f(P)= \frac 12 \|F(P)\|^2_F$}

\label{page:motivexactsol}
Since the function $\sigma_2(P)$ is nonconvex, and optimization methods
generally find local minima,
we investigate the possibility for {such a quartic function to have all local minimizers as global minimizers. 
This question has been considered} as open and has been widely studied in  the MDS literature; {see for example,}
\cite{MaloneTrosset2000b,MaloneTarazagaTrosset2002,Trosset97,Parhizkar:196439}. 
One of the motivations for the research about  the local nonglobal minima ({\bf lngm})
of $\sigma_2(P)$ is that the characterization of the \lngm is critical to developing efficient algorithms 
without resorting to convex (semidefinite programming) relaxations.


The question about the existence of a \lngm
was also considered for another type of stress function,
called the \emph{raw stress}:
\index{raw stress, $\sigma_1(P)$} 
\index{$\sigma_1(P)$, raw stress} 
$$\sigma_1(P) = \sum_{i=1}^n \sum_{j=1}^n \left( \| p_i - p_j \| - \| \bar{p}_i - \bar{p}_j \| \right)^2.$$
Trosset and Mathar  \cite{Trosset97} analytically verified that the raw
stress function $\sigma_1(P)$ admits a \lngm, where $\bar P$ is a square
configuration with vertices at the four points $\bar p_1=[1/2, 1/2], \bar p_2=[-1/2, 1/2], \bar p_3=[-1/2, -1/2]$, and $\bar p_4=[1/2, -1/2]$. 
The authors of \cite{Zilinskas03} applied this
example to the smooth function $\sigma_{2}(P)$ but did not find
a \lngmp. Instead, they present an example of the \emph{inexact} \EDM recovery
problem having a \lngmp; specifically, problem \cref{eq:s2P} with $\bar D$ replaced by $\Delta\in\Sn$
that is \emph{not} an \EDM. However, the question about the existence of a \lngm for the exact problem
$\sigma_{2}(P)$ remained open. Addressing this challenge involves two main difficulties: 
(1)  the apparent lack of simple examples exhibiting \lngm ; and (2) the
inherent complexity in proving the existence of a \lngm for such a
complex problem.

In this paper, we provide a definitive answer to this open question. 
{We prove that all second-order stationary points are global minimizers whenever $n \leq d + 1$. 
For $n > d+1$, we present an example in dimension $d=1$, with a very special structure, 
for which we can analytically exhibit a \lngm. 
For more general cases,} 
we find examples where the function $\sigma_{2}(P)$ has a \lngmp, and we provide analytic verification
using techniques from the local convergence proof of Newton's method.
The examples are obtained using the trust region approach with random initializations.

The rest of this paper is arranged as follows.
We continue in~\Cref{sect:mainprob} with a description of our main
unconstrained minimization problem \cref{eq:s2P}. We introduce two additional equivalent
problems with reduced numbers of variables. The reduction allows for strict
second-order sufficient optimality conditions and thus is necessary for the analytic existence proof.
In~\Cref{sect:properties}, we include various linear transformations,
derivatives, and adjoints. Many of these are used throughout this paper.
We suggest that they provide a useful addition to the literature on \EDMp, 
as it emphasizes the use of matrix transformations rather than individual elements or points.
In \Cref{sect:optcondsens}, we study the optimality conditions and establish that $\sigma_{2}(P)$ has no
\lngmp\ if $n \leq d+1$. {In \Cref{sect:analytexlngm}, we give a
special example with $d=1$ where a \lngmp\ can be explictly illustrated}. 
In \Cref{sect:KantSens}, we provide two examples and use
the Kantorovich theorem to (numerically) prove the
existence of \lngmp s {in this more general setting}.

\section{Notation and {Equivalent} Main Problem Formulations}
\label{sect:mainprob}
Before presenting the problem  formulations, we introduce the necessary notation and background from distance geometry.
Further details are provided in \cite{alfm18}.

\subsection{Notation}
\label{sect:notat}
We use the
\textdef{trace inner product} in $m\times n$ matrix space $\Rmn$,
 $\langle X,Y\rangle = \trace X^TY$ with the induced  
 	\textdef{Frobenius norm, $\|X\|_F=\sqrt {\langle X, X\rangle}$}.  \index{$\|X\|_F$, Frobenius norm}
Denote $\|X\|_2:=\sqrt{\lambda_{\max}(X^TX)}$, where
$\lambda_{\max}(\cdot)$ gives the largest eigenvalue. If no subscript in
$\|\cdot\|$ is written, the Frobenius norm $\|\cdot\|_F$ is understood.
We note that both $\|\cdot\|_F$ and $\|\cdot\|_2$ reduce to the standard 
$2$-norm in $\Rn$ ($n\times 1$ matrices) $\|x\|$. 
For a finite dimensional Hilbert space ${\cal X}$,
we use $B_r(\tilde x):= \{ x \in {\cal X} \mid \| x - \tilde x\|
\leq r \}$ to denote the ball centered at $\tilde x$ with radius $r>0$.

Let $\Sn$ be the set of symmetric matrices in $\R^{n\times n}$. 
The cone of positive semidefinite matrices is denoted by $\Snp \subset
\Sn$, and we use $S\succeq 0$ for $S\in \Snp$. Similarly, for positive
definite matrices $S$, we use $S\in\Snpp$ and $S\succ 0$. 
Additionally, {$S \geq 0$ and $S>0$} denote that all entries of $S$ are
non-negative and positive, respectively.
Let \textdef{$\diag(S)\in \Rn$} denote
the vector formed by the diagonal of a matrix $S\in \Rnn$. The adjoint operator
\textdef{$\diag^*(v)=\Diag(v)\in \Sn$} maps a vector  $v\in \Rn$ to $\Diag(v) \in \Sn$, the diagonal matrix with entries from $v$.
For a matrix $C \in \Rnd$, $\kvec (C) \in \R^{nd}$ denotes the column vector formed by stacking 
the columns of $C$, and $\kMat \cong \kvec^*$ is the
adjoint of $\kvec$, satisfying $\kMat(\kvec(C)) = C$ for all $C \in \Rnd$.
\label{page:Frechet}If $F: {\cal X} \rightarrow {\cal Y}$ is a map between finite
dimensional Hilbert spaces, $F'(P)$ and $F''(P)$ denote its Fr\'echet
derivatives at $P \in {\cal X}$. 

\index{$\Snp$, positive semidefinite cone}
\index{positive semidefinite cone, $\Snp$}
\index{$\Snpp$, positive definite cone}
\index{positive definite cone, $\Snpp$}

\index{$\|\cdot\|_F$, Frobenius norm}
\index{Frobenius norm, $\|\cdot\|_F$}
\index{$e$, vector of ones}
\index{vector of ones, $e$}
\index{embedding dimension, $d$}
\index{$d$, embedding dimension}
\index{configuration matrix, $\bar P$}
\index{$\bar P$, configuration matrix}

For a set of points $p_i\in \R^d,\ i\in [n]:=\{1, 2, \cdots, n\}$, denote
the \emph{configuration matrix} by $$P=\begin{bmatrix} p_1\  
	p_2\  \ldots\  p_{n} \end{bmatrix}^T\in \R^{n\times d}.$$ Here $d$ is the embedding dimension. 
Denote the quadratic mapping $\cM : \Rnd \to \Sn$, \textdef{$\cM(P) = PP^T$}.
\index{Euclidean distance matrix, \EDMp} 
\label{page:Pcentered}
Recall that $e$ denotes the column vector of ones of appropriate dimension. 
Then, the classical result of
Sch\"oenberg~\cite{Sch35} relates an \EDM, $\cD(P)$,
with the corresponding \textdef{Gram matrix, $G=\cM(P)$}, by applying the
linear operator $\KK: \SC \to \SH$: 
\begin{equation}
	\label{eq:Dbardata}
 \KK(G) = \diag(G)e^T+e\diag(G)^T-2G =  {\left(\|p_i-p_j\|^2\right)_{ij} =:
	\cD(P),} 
\end{equation}\label{page:2.1}
where the 
\textdef{centered subspace, $\SC$} and the \textdef{hollow subspace, $\SH$} are defined by 
\index{$\SC$, centered} 
\index{$\SH$, hollow} 
\[
\SC = \{S\in \Sn \,:\, Se = 0\}, \quad
\SH = \{S\in \Sn \,:\, \diag(S) = 0\}.
\]
Denote $S_e: \Rn \to \Sn$:  
\textdef{$S_e(v) = ve^T+ev^T$}. Then, $ \KK(G)=S_e(\diag(G))-2G$.
Note that the centered assumption $P^Te=0 \Leftrightarrow G=PP^T\in \SC$.
Also, when the domain of $\KK$ is restricted to be $\SC$, the mapping $\KK$ is a bijection between $\SC$ and $\KK(\SC)$.

\index{\EDMp, Euclidean distance matrix}
\index{$\KK(G)$}

Further detailed properties and a list of (non)linear transformations
and adjoints are given in~\Cref{sect:properties}.

\subsection{Main Problem Formulations}

Suppose that we are given a centered configuration matrix $\bar P\in
\Rnd, \bar P^Te=0$. 
This gives rise to the corresponding Gram
matrix $\bar G = \bar P\bar P^T\in \SC$ and \EDMp, 
	$\bar D  = \cK(\bar G)$.
We now present the main problem and two reformulations that reduce the size of variables and help with stability.

\index{\lngmp, local nonglobal minimum}

\begin{problem}
	\label{prob:origF}
	Let $\bar D=	\cD(\bar P)$ be a \EDM obtained from some
given configuration matrix $\bar P$.
	Consider the 
	\emph{nonconvex} minimization problem of \underline{recovering}
	a corresponding  configuration matrix $\hat{P}$ given by  
	\begin{equation}\label{eq:F}
		\hat P \in \argmin\limits_{P\in \Rnd} 
		\textdef{$f(P) := \frac 12 \|\cK(PP^T) - \bar D\|_F^2$}
		=:\frac 12 \|F(P)\|^2_F;
	\end{equation}
	thus defining the function \textdef{$F : \Rnd \to \SH$}.
\end{problem}

\index{$t(n-1):=n(n-1)/2$, triangular number}

\Cref{prob:origF}
is a \underline{non}linear least squares problem. 
It has $nd$ variables.
By taking advantage of symmetry and the zero-diagonal constraints, 
the objective function can be seen as
a sum of squares of $t(n-1)$ quadratic functions, where  $t(n-1):=n(n-1)/2$ is the \textdef{triangular number}.
Note that $\bar P$ is a global minimizer for
\cref{eq:F}
with the optimal value $f(\bar P) = 0$.
We  study whether all stationary
points where the second-order necessary optimality conditions hold are global minimizers. 

\index{$t(n-1)=n(n-1)/2$, triangular number} 
\index{\lngmp, local nonglobal minimum}

Note that the distance matrix is invariant under translations and
rotations of $P$. Without loss of generality, we assume $P$ is centered ($P^Te = 0$).
\index{embedding dimension, $d$}
Let $V \in \R^{n \times (n-1)}$ be such that
\begin{equation}
	\label{eq:Vmatr}
	 V^TV=I_{n-1}, \quad V^Te = 0.
\end{equation}
By the fact that $VV^T$ is the orthogonal projection onto $e^\perp$ (the orthogonal complement of $e$),   
\label{page:PVL}
we have $P^Te=0$ if, and only if, $P=VL$ for some $L\in
\R^{(n-1) \times d}$.\footnote{This is similar to the application of
\textdef{facial reduction} for the semidefinite relaxation, 
see~\cite{AlKaWo:97}.}  We exploit this property for deducing an equivalent problem formulation having a smaller dimension.
\begin{problem}
	\label{prob:reducedmainptrecover}
Let $\bar P, \bar D$ be as given in
\Cref{prob:origF}, and let $V$ be as in~\cref{eq:Vmatr}. Consider the 
\emph{nonconvex} minimization problem of \underline{recovering}
a corresponding  centered configuration matrix $\hat{P}=V\hat{L}$ by
finding
	\begin{equation}
		\label{eq:VL}
		\hat L \in \argmin_{L\in \R^{(n-1)\times d} }
		\textdef{$f_L(L) := \frac 12 \|\cK(VL(VL)^T) - \bar D\|_F^2$} =:\frac 12
		\|F_L(L)\|^2_F;
	\end{equation}
	thus defining the function \textdef{$F_L : \R^{(n-1)\times d} \to 
		\SH$}. 
\end{problem}
\index{$f_L(L)$}

\index{$\cO = \{Q\in \Rdd : Q^TQ=I_d\}$, orthogonal group of order $d$}
\index{orthogonal group order $d$, $\cO = \{Q\in \Rdd : Q^TQ=I_d\}$}

Let $\cO = \{Q\in \Rdd : Q^TQ=I_d\}$ be the orthogonal group {of order $d$}.
Note that $LL^T = LQQ^TL^T$ holds for all $Q \in \cO$. 
If $L^T = QR$ is the QR factorization, then $$R^T = LQ \Rightarrow f_L(L) = f_L(R^TQ^T) = f_L(R^T),$$ where $R\in \R^{d\times (n-1)}$ is upper triangular (trapezoidal when $d<n-1$).
The problem can be further reduced using rotation invariance: $f_L(LQ)=
f_L(L), \, \forall Q\in \cO$. 

\label{page:Ltriagdef}
Recall that the linear transformation $\svec: \Sn \to \R^{t(n)}$ is a
generalization of the vectorization $\kvec$ applied to symmetric
matrices that avoids the duplication of the lower triangular part.
We now extend this idea to triangular (trapezoidal) matrices to avoid the
zeros.
We define the linear operator that maps a vector  $\ell
\in \R^{t_{\ell}}$ to a lower triangular (trapezoidal)
matrix in $\Rnod$ given by
\begin{equation}
	\label{eq:ltriagdef1}
 \textdef{$\LTriag(\ell )$}_{(i,j)}=\left\{
	\begin{array}{cl}
		\ell_{nj-n-t(j)+i+1}, & \text{if  } j\le i \\
		0, & \text{otherwise,}
	\end{array}
	\right.
\end{equation}
where 
\index{$t_\ell$, in \cref{eq:tellzeros1}}
\begin{equation}\label{eq:tellzeros1}t_\ell = \left\{
	\begin{array}{cl}
		t(n-1), & \text{if  } d\geq n-1 \\
		(n-1)d-t(d-1), & \text{otherwise.}\\
	\end{array}
	\right.
\end{equation}
For $L\in\Rnod$ lower trapezoidal, $t_\ell$ counts its entries where nonzero values are allowed. 
For $d < n-1$,  \textdef{$\LTriag(\ell )$} has $t(d-1)$ {zero elements} at the top right;
whereas for $d\geq n-1$,  \textdef{$\LTriag(\ell )$} has $t(n-1)$ {nonzero elements} at the
bottom left.

Using the above definition, we define
\begin{equation}
	\label{eq:flell1}
	f_\ell : \R^{t_\ell }  \to \R, \quad 
	\textdef{$f_\ell(\ell) := f_L ( \LTriag(\ell))$}. 
\end{equation}
Notice that the adjoint of $\LTriag$, $\LTriag^* : \Rnod \to \R^{t_\ell} $, takes the lower triangular (trapezoidal) part of $L \in \Rnod$ and maps it to the corresponding vector $\ell\in \R^{t_\ell }$, such that $\LTriag^*\LTriag(\ell) = \ell$. Moreover, $\LTriag \LTriag^* (L)$ is the projection of $L$ onto the subspace of lower triangular (trapezoidal) matrices.

\begin{problem}  
\label{prob:reducedmainptrecovertriang}
Let $\bar P, \bar D$ be as given in \Cref{prob:reducedmainptrecover} (and in
\Cref{prob:origF}), and let $V, t_\ell, f_\ell(\ell)$,  be as 
	in~\cref{eq:Vmatr,eq:tellzeros1,eq:flell1}, respectively.
	Consider the 
	\emph{nonconvex} minimization problem of \underline{recovering}
	a corresponding  centered configuration matrix 
	$\hat{P}=V\hat{L}=V\LTriag(\hat{\ell})\hat{Q}^T$, with $\hat{Q}
\in \cO$, by finding
	\begin{equation}
		\label{eq:VLtriang}
\begin{array}{rcl}
		\hat \ell \in \argmin_{\ell \in \R^{t_\ell }}
		f_{\ell}(\ell) 
&:=&
 \frac 12 \|\cK(V\LTriag(\ell) (V\LTriag(\ell) )^T) - \bar
		D\|_F^2 
\\&& \qquad =\frac 12 \|F_{L}(\LTriag(\ell) )\|^2_F =: \frac 12
\|F_{\ell}(\ell)\|^2_F;
\end{array}
    \end{equation}
    thus defining the function \textdef{$F_\ell : \R^{t_{\ell}} \to \SH$}.
\end{problem}

\begin{remark}
	\label{rem:mainreducprobtriang} 
	Compared to \Cref{prob:origF}, \Cref{prob:reducedmainptrecovertriang} is a nonlinear least
	squares problem, with fewer variables and the same number of quadratic {terms}  
	$\left(\cK(V\LTriag(\ell) (V\LTriag(\ell) )^T) - \bar D\right)_{ij}, i<j$.
	For $d < n - 1$, the underlying system of
equations is overdetermined, as $t_\ell < t(n-1)$. For $d \geq n-1$,
from \cref{eq:tellzeros1}, the number of variables is $t(n-1)$, the same
as the number of quadratic equations. {Thus we no longer have the
singularity that arises for the Jacobian of an underdetermined nonlinear
least squares problem.}
\end{remark}

In order to determine whether a
\lngm exists, the next section provides useful formulae for
linear transformations and derivatives.

\section{Properties and auxiliary results}  \label{sect:properties}

We now provide appropriate notation and formulae for transformations,
adjoints and derivatives involved in \EDMp, and then give the
equivalence relationships among local minimizers of the above three reformulations.

\subsection{Transformations, Derivatives, Adjoints, Range and Null Spaces}
\label{sect:transderadj}

\Cref{lem:adjderivs} below presents a list of auxiliary results. 
{It concerns the following vectors,
matrices and functions:}
\[
\begin{array}{c}
	P\in \Rnd,\, p=\kvec(P) \in \R^{nd}, \Delta P\in \Rnd,\, \Delta p=\kvec(\Delta P) \in \R^{nd}, 
	\\ L\in \Rnod,\ \ell\in\R^{t_\ell},t_\ell \text  {\ in
	   \cref{eq:tellzeros1}},\, S, T\in \Sn;
\end{array}
\]
\[
\begin{array}{c}
	\cM : \Rnd \to \Sn,\, \cK : \Sn\to\Sn,\, F: \Rnd \to \SH,\,f:\Rnd \to \R, 
	\\
	\LTriag : \R^{t_\ell}  \to \Rnod,\  S_e: \R^n\to\Sn,\  F_L : \Rnod \to \SH,\  f_L:\Rnod \to \R, 
	\\
	\ F_\ell:\R^{t_\ell}\to\SH,\  f_\ell:\R^{t_\ell}\to\R.
\end{array}
	\]

\begin{lemma}	\label{lem:adjderivs} We have the following  first and second Fr\'echet derivatives and adjoints:
	\begin{enumerate}
		\item
\label{item:firsstorderexp}
		\textdef{$\cM^\prime(P)(\Delta P) = P\Delta P^T+\Delta P P^T$}, 
		$\cM^{\prime\prime}(P)(\Delta P, \Delta P) =2\Delta P \Delta P^T$.

		\item\label{item:Mpadj}
		\textdef{$\cM^\prime(P)^*(S) = 2SP$}.

		\item
\label{item:SeSSe}
		\textdef{$S_e^*(S)= 2Se$}.

		\item
		\label{item:cK}
		\textdef{$\cK(G) = S_e(\diag(G)) - 2G$}, $\range(\cK) = \SH, \, \nul(\cK) = \range(S_e)$.

		\item  
		\label{item:LaplKsNinvar}
		$\cK^*(S)=2(\Diag (Se)-S)$,  
		$\range(\cK^*) = \SC,\,\nul(\cK^*) = \Diag(\Rn)$. Moreover, 
		$S\geq (\leq) 0 \implies \cK^*(S) \succeq (\preceq) 0$.

		\item
		\textdef{$\cD(P)= S_e(\diag(\cM(P)))-2\cM(P)$.}
		
		\item
		\textdef{$ \cD^\prime(P)(\Delta P) = S_e(\diag(\cM^\prime(P)(\Delta
			P))-2\cM^\prime(P)(\Delta P)$}. 
		
		\item
		\label{item:firstsecondderivF}
		\textdef{$F^\prime(P)(\Delta P) = \cK\left(\cM^\prime(P)(\Delta P)\right)$}, 
		\textdef{$F^{\prime\prime}(P)(\Delta P,\Delta P) =
			\cK\left(\cM^{\prime\prime}(P)(\Delta P,\Delta P)\right)$}.

		\item\label{item:Fpadj} 
		$F^\prime(P)^*(S) =\cM^\prime(P)^*(\cK^*(S)) =4(\Diag (Se)-S)P$.

		\item
\label{item:symmLT}
We have  
		\begin{equation}
			\label{eq:1storderHess}
	f^\prime(P)=F^\prime(P)^*(F(P)) =4[\Diag (F(P)e)-F(P)]P,
		\end{equation}
		and 
\begin{equation}
\label{eq:fprpr}
\begin{array}{rcl}
f^{\prime\prime}(P)(\Delta P, \Delta P) & = & \langle \cK(P\Delta P^T+\Delta P P^T), \cK(P\Delta	P^T+\Delta P P^T)\rangle \\
& & \qquad +2\langle F(P), \cK(\Delta P \Delta P^T)\rangle.
\end{array}
\end{equation}
		\index{$\cS(K)=(K+K^T)/2$}
	\end{enumerate}
\end{lemma}

\begin{proof}
\begin{enumerate}
\item 		It follows directly from the expansion 
		$$
		\begin{array}{rcl}
			\cM(P+\Delta P)
			&=&
			(P+\Delta P)(P+\Delta P)^T\\
			&=&
			PP^T+\Delta P P^T+ P\Delta P^T+ \Delta P\Delta P^T
			\\&=&
			\cM(P)+\cM^\prime(P)(\Delta P) + \frac 12 
			\cM^{\prime\prime}(P)(\Delta P, \Delta P) .
		\end{array}
		$$

\item 	Note that	\[
		\begin{array}{rcl}
			\langle  \cM^\prime(P)(\Delta P),S \rangle 
			&=&
			\langle  P\Delta P^T+\Delta P P^T,S \rangle 
			\\&=&
			\trace (P\Delta P^TS +\Delta P P^TS)
			\\&=&
			\trace (SP\Delta P^T +SP\Delta P^T)
			\\&=&
			\langle  \cM^\prime(P)^*(S),\Delta P \rangle.
		\end{array}
		\]
\item 	From the trace inner product,	
\[
		\langle S_e(v),S\rangle = \trace (ev^TS+ve^TS) =  \trace(v^TSe) +
		\trace(Sev^T) = \langle 2Se,v\rangle.
		\]

\item See~\cite[Prop. 2.2]{homwolkA:04}.

\item See~\cite[Prop. 2.2]{homwolkA:04} for the characterization of the nullspace of $\cK^*$. 
		\index{$\cK^*(S)=2(\Diag (Se)-S)$, Laplacian matrix} 
		\index{Laplacian matrix, $\cK^*(S)=2(\Diag (Se)-S)$} 
		The identity $\cK^*(S)=2(\Diag (Se)-S)$ follows from 
		\begin{equation*}
			\begin{array}{ccl}
				\langle \cK(T), S\rangle &=&\langle\diag(T)e^T+e\diag(T)^T-2T
				, S\rangle\\
				&=&2\trace(e^TS \diag(T))-2\trace(TS)\\
				&=&2\langle Se, \diag(T)\rangle-2\langle T, S\rangle\\
				&=&
				2\langle \Diag (Se)-S, T\rangle,
			\end{array}
		\end{equation*}
		where the last equality is due to $\Diag = \diag^*$. Moreover, for $S\in \Sn$, we have by diagonal dominance that
		$S\geq (\leq) 0 \implies \cK^*(S) \succeq (\preceq) 0$.

\item It follows directly from \Cref{item:LaplKsNinvar}.

\item This follows from the linearity of $\diag$ and $S_e$.

\item  Both follow from the definitions and linearity of $\cK$.

\item  
\label{iterm:FpPSMpPKS}
It follows from 
\begin{equation*}
			\begin{array}{ccl}
				\langle F^\prime(P)(\Delta P), S\rangle &=&\langle \cK (\cM^\prime(P)(\Delta P)), S\rangle\\
				&=&\langle  \cM^\prime(P)(\Delta P), \cK^*(S)\rangle\\
				&=&\langle \Delta P, \cM^\prime(P)^*(\cK^*(S))\rangle,
			\end{array}
		\end{equation*}
and
$\cM^\prime(P)^*$ and $\cK^*(S)$ presented in 
\Cref{item:Mpadj,item:LaplKsNinvar}.

\item 		From the expansion of $f(P+\Delta P)$, 
		\begin{equation*}
			\begin{array}{ccl}
				&&f(P+\Delta P)\\
				&=&\frac{1}{2}\langle F(P+\Delta P), F(P+\Delta P)\rangle\\
				&=&\frac{1}{2}\| F(P)+F^\prime(P)(\Delta P)+\frac{1}{2}F^{\prime\prime}(P)(\Delta P, \Delta P)+o({\|\Delta P\|}^2)\|^2,\\
				&=&\frac{1}{2}\langle F(P), F(P)\rangle+\langle F(P), F^\prime(P)(\Delta P)\rangle\\
				&&+\frac{1}{2}\langle F^\prime(P)(\Delta P), F^\prime(P)(\Delta P)\rangle+\frac{1}{2}\langle F(P), F^{\prime\prime}(P)(\Delta P, \Delta P)\rangle+o(\|\Delta P\|^2),
			\end{array}
		\end{equation*}
we get \cref{eq:1storderHess}. Then
		we obtain 
		\begin{equation}
			\label{eq:Hessiantwoterms}
			\begin{array}{ccl}
				&& f^{\prime\prime}(P)(\Delta P, \Delta P)\\
				&=&\langle F^\prime(P)(\Delta P), F^\prime(P)(\Delta P)\rangle+\langle F(P), F^{\prime\prime}(P)(\Delta P, \Delta P)\rangle\\
				&=&\langle \cK(\cM^\prime(P)(\Delta P)), \cK(\cM^\prime(P)(\Delta P))\rangle+\langle F(P), \cK(\cM^{\prime\prime}(P)(\Delta P, \Delta P))\rangle
				\\
				&=&
				\langle \cK(P\Delta P^T+\Delta P P^T), \cK(P\Delta
				P^T+\Delta P P^T)\rangle+2\langle F(P), \cK(\Delta P \Delta P^T)\rangle,
			\end{array}
		\end{equation}
where the second equality follows from \Cref{item:firstsecondderivF}. Define $\Delta p := \kvec(\Delta P)$. 
	Now, we can isolate the matrix representation with 
$$\begin{array}{ccl}
				f^{\prime\prime}(P)(\Delta P, \Delta P)
				&=&
				\langle f^{\prime\prime}(P)(\Mat \kvec (\Delta P)), 
				\Mat \kvec (\Delta P) \rangle
				\\&=&
				\langle \left[\kvec f^{\prime\prime}(P)\Mat\right] \left( \Delta p\right), \left(\Delta p\right) \rangle.
			\end{array}$$
Denote the symmetrization
$\cS: \Rnn\rightarrow\Sn, \cS(K)=(K+K^T)/2$, and let
$\cT$ be the self-adjoint transpose operator. The first term in~\cref{eq:Hessiantwoterms} is
		\begin{equation}
			\label{eq:firstterm}
			\begin{array}{ccl}
				&&
				4\langle \cK(\cS(P(\Mat \kvec \Delta P)^T)),\cK(\cS(P(\Mat
				\kvec \Delta P)^T)) \rangle\\
				&=&
				4\langle (P^T\cS^*\cK^*\cK\cS P)((\Mat \kvec \Delta P)^T),(\Mat
				\kvec \Delta P)^T \rangle\\
				&=&4\langle (P^T\cS^*\cK^*\cK\cS P)(\cT \Mat \kvec \Delta P),(\cT \Mat
				\kvec \Delta P) \rangle\\
				&=&
				4\left\langle \left[\kvec \cT^*P^T\cS^*\cK^*\cK\cS P \cT \Mat
				\right] \Delta p,
				\Delta p \right\rangle.
			\end{array}
		\end{equation}
		The second term in~\cref{eq:Hessiantwoterms} is
		\begin{equation}
			\label{eq:secondterm}
			\begin{array}{ccl}
				&&2\langle F(P), \cK\left(\Delta P  \Delta P^T\right)\rangle\\
				&=&
				2\left\langle \cK^*\left(F(P)\right),\Delta P \Delta P^T\right\rangle
				\\&=&2\left\langle \Delta P, \cK^*(F(P)) \Delta P\right\rangle
				\\&=&2\left\langle  \left[\kvec \cK^*F(P)\Mat \right] 
				\Delta p, \Delta p\right\rangle.
			\end{array}
		\end{equation}
		Recall that $F^\prime(P)(\Delta P) = \cK(\cM^\prime(P)(\Delta P))$.
		We combine~\cref{eq:firstterm,eq:secondterm} and obtain
		the matrix representation of the Hessian (not necessarily positive
		semidefinite) :
		\begin{equation}
			\label{eq:2ndorderHess}
			\begin{array}{rcl}
				\left[\kvec f^{\prime\prime}(P)\Mat\right]
				&=&
				4\left[\kvec \cT^*P^T\cS^*\cK^*\cK\cS P \cT \Mat \right] 
				\\&& \qquad  \qquad + 2\left[\kvec \cK^*F(P)\Mat \right] 
				\\&=&
				4\left[J^* J \right] + 2\left[\kvec \cK^*F(P)\Mat \right],
			\end{array}
		\end{equation}
		where 
		\begin{equation}
			\label{eq:dJ}
			J (\Delta p ):= \cK \cS P \cT \Mat \Delta p.
		\end{equation}
\end{enumerate}
\end{proof}


\begin{theorem}\label{thm:3.2}
   The second-order necessary optimality conditions for 
	\cref{eq:F} are:
	\begin{equation}\label{eq:1opt}
		\begin{array}{rcl}
	0 &=& f'(P) = F^\prime(P)^*(F(P))= 2\cK^*(F(P))P,
		\end{array}
	\end{equation}
	\begin{equation}\label{eq:2opt}
		\begin{array}{rcl}
			0 &\preceq &
			\left[\kvec f^{\prime\prime}(P)\Mat\right] = 4\left[J^*J \right] + 2\left[\kvec \cK^*(F(P))\Mat \right].
		\end{array}
	\end{equation}
\end{theorem}
\begin{proof}
\label{page:f'}
{The second equality in \cref{eq:1opt} follows from \cref{eq:1storderHess}, and the third equality in \cref{eq:1opt} follows from 
\Cref{item:Fpadj,item:Mpadj} of \Cref{lem:adjderivs}.} In
particular, we have
\begin{equation}\label{eq:rev1}
	f'(P) = F'(P)^*(F(P)) = \cM'(P)^*(\cK^*(F(P))) = 2 \cK^*(F(P))P.
\end{equation}
	The expression for the second-order term in \cref{eq:2opt}
follows from \cref{eq:2ndorderHess} in
\Cref{item:symmLT} of \Cref{lem:adjderivs}.
\end{proof}
Throughout the paper, we denote the following two matrices in $\mathbb{S}^{nd}$: 
\begin{equation}
	\label{eq:H12}  
	\textdef{$H_1=[J^* J]$},\  \textdef{$H_2=[\kvec
		\left(\cK^*\left(F(P\right)\right)\Mat]$}.
\end{equation}
{By abuse of notation, $H_1$ and $H_2$ represent both the linear
maps and their matrix representations. The meaning will be clear from
the context.}\label{page:H1H2}
We call $P$ a \textdef{stationary point} if \cref{eq:1opt} holds, and we
call $P$ a \textdef{second-order stationary point} if both \cref{eq:1opt} and \cref{eq:2opt} hold. 
	\subsection{Optimality Conditions of Three Problem Formulations}
	According to the chain rule, the derivatives and  optimality conditions  of $f_L$ defined in \cref{eq:VL} and $f_\ell$ defined in \cref{eq:VLtriang} can be easily obtained from that of $f$.
	\begin{proposition} 
\label{pro:fL''}
		The derivatives of $f_L(L):\R^{(n-1)\times d}\to \R$ are
		\begin{equation}\label{eq:1optfv0}
			\begin{array}{rcl}
				f'_L(L) =V^Tf'(VL),
			\end{array}
		\end{equation}
		and
		\begin{equation}\label{eq:2optfv0}
			\begin{array}{rcl}
				&&
				f''_L(L) = V^Tf''(VL)V.
			\end{array}
		\end{equation}
	\end{proposition} 
	
	\begin{proposition} 
\label{pro:flpandpp}
		The derivatives of $f_\ell(\ell):\R^{t_{\ell}}\to \R$ are
		\begin{equation}\label{eq:1optfl}
			\begin{array}{rcl}
	f'_\ell(\ell) =\LTriag^*f'_L(\LTriag(\ell)),
			\end{array}
		\end{equation}
		and
		\begin{equation}\label{eq:2optfl}
			\begin{array}{rcl}
				&&
				f''_\ell(\ell) = \LTriag^*f''_L(\LTriag(\ell))\LTriag.
			\end{array}
		\end{equation}
	\end{proposition} 
	
In the following, we show that any local minimizer of \cref{eq:VL}
corresponds to a family of local minimizers of \cref{eq:F}, obtained by
translations.  Similarly, any local minimizer of \cref{eq:VLtriang} 
corresponds to a family of local minimizers of \cref{eq:VL}, derived from
rotations.


\begin{proposition}  \label{th:ellL}
The configuration matrix ${P_*}\in\R^{n\times d}$ is a local minimizer of the function~$f$ (see \cref{eq:F}) if,
and only if, all configurations in $\{P_* + ev^T: v\in\R^d\}$ are local minimizers of the function $f$. 
\end{proposition}

\begin{proof}
We exploit the fact that the function $f$ is invariant w.r.t.~translations:
for any point $P$, we have
\begin{displaymath}
f(P) = f(P + ev^T) .
\end{displaymath}
If $P_*$ is a local minimizer of $f$, there must be a $\delta > 0$ such that:
\begin{displaymath}
\forall P \, : \, || P_* - P ||_F \le \delta ,\quad f(P_*) \le f(P).
\end{displaymath}
Then, for all $\hat P$ such that $\| \hat P - (P_* + ev^T) \|_F \leq
\delta$, 
we have $\| (\hat P - ev^T) - P_* \|_F \leq \delta$, thus
\begin{displaymath}
f(\hat P) = f(\hat P - ev^T) \geq f(P_*) = f(P_* + ev^T)
\end{displaymath}
implying that $P_* + ev^T$ is also a local minimizer. The other implication
follows similarly.
\end{proof}

\begin{proposition}  \label{th:eL}
The configuration matrix $L_*\in\R^{(n-1)\times d}$ is a local minimizer
of the function~$f_L$ (see \cref{eq:VL}) if,
and only if, all configurations in $\{L_*Q: Q\in\mathcal{O}\}$ are local minimizers of $f_L$.
\end{proposition}

\begin{proof}
We now exploit the fact that the function $f_L$ is invariant w.r.t.~rotations:
for any configuration $L$, and $Q \in \cO$, we have
\begin{displaymath}
f_L(L) = f_L(LQ) .
\end{displaymath}
If $L_*$ is a local minimizer of $f_L$, there must be a $\delta > 0$ such that:
\begin{displaymath}
\forall L \, : \, || L_* - L ||_F \le \delta ,\quad f_L(L_*) \le f_L(L).
\end{displaymath}
Then, for all $\hat L$ such that $\|\hat L - L_* Q \|_F \leq \delta$, 
we have 
\begin{displaymath}
\| \hat L Q^T - L_* \|_F = \|\hat L - L_* Q \|_F \leq \delta, 
\end{displaymath}
thus
\begin{displaymath}
f_L(\hat L) = f_L(\hat L Q^T) \geq f_L(L_*) = f_L(L_* Q)
\end{displaymath}
implying that $L_*Q$ is also a local minimizer. The other implication follows similarly.
\end{proof}

The local minimizers of the two functions in equations~\cref{eq:F} and~\cref{eq:VL} have the following relationships.

	\begin{theorem}\label{th:equi2}
		Let $P_*\in\R^{n\times d}$ and $V$ be as defined in \cref{eq:Vmatr}. 
		Denote
		\[
		v_*=\frac 1n P_*^{T}e\in\R^d,\, P_{v_*}=P_*-ev_*^{T},\,  L_*=V^TP_{v_*}.
		\]
		Then, $L_*$ is a local
		minimizer of \cref{eq:VL} if, and only if, $P_{v_*}$ and $P_*$ are local minimizers of \cref{eq:F}. 
	\end{theorem}
	\begin{proof}
	 First, recall that $VV^T$ is the orthogonal projection onto $e^\perp$ and that the
		columns of $P_{v_*}$ are centered. Thus,~we have $VL_*=VV^TP_{v_*} = P_{v_*}$. 
		Sufficiency:	Let $P_{v_*}$ be a local minimizer of \cref{eq:F}. Then, there exists $\delta>0$ such that 
		\begin{equation}\label{eq:PP*}
			f(P_{v_{*}})\le f(P),\ \forall P \ : \ \|P-P_{v_{*}}\|_F\le \delta.
		\end{equation}
		For any $L\in\R^{(n-1)\times  p}$ such that $\|L-L_*\|_F\le\delta$, let $\hat{P} = VL$. Then, we have
		\begin{equation*}
		f_L(L_*)=f(VL_*) = f(P_{v_{*}}) \le f(\hat{P}) = f(VL)=f_L(L),
		\end{equation*}
		where the inequality is due to $\|\hat{P} - P_{v_{*}} \|_F= \|VL-VL_*\|_F=\|L-L_*\|_F\le\delta$ and \cref{eq:PP*}, and the equalities hold by the definition of $f_L$.
		\\Necessity: 
		Suppose $L_*$ is a local minimizer of $f_L(L)$, meaning there exists $\delta>0$ such that
		\begin{equation}\label{eq:11}f_L(L_*)\le f_L(L),\ \forall L \ : \ \|L-L_*\|_F\le \delta.
		\end{equation}
		For any configuration $P$ with $\|P-P_{v_{*}}\|_F\le \delta$, define its centroid $v=P^{T}e/n$. Then, there exists $L\in\R^{(n-1)\times d}$ such that the centered configuration can be expressed as 
$P=VL+ev^{T}$. This implies that $P-P_{v_{*}}=V(L-L_*)+ev^{T}$. As $V(L-L_*)$ and $ev^{T}$ are orthogonal, we get
		\begin{equation}\label{eq:deltakey}
		\|L-L_*\|_F^2=\|V(L-L_*)\|_F^2=\|P-P_{v_{*}}\|_F^2-\|ev^{T}\|_F^2\le \delta^2.
		\end{equation}
		Now, from \cref{eq:11} and \cref{eq:deltakey}, we have 
		\begin{equation*}
			f(P)=f(VL+ev^{T})=f(VL)\ge f(VL_*)=f(P_v{_*}),
		\end{equation*}
implying that $P_v{_*}$ is a local minimizer of $f(P)$. According to \Cref{th:ellL}, ${P_*}$ is also a local minimizer of $f(P)$.
	\end{proof}
From \Cref{th:eL}, we know that for the case of $d\ge 2$, if $L$ is a local minimizer of $f_L(L)$, then $\{LQ: Q \in \cO\}$ is a local minimizer of $f_L(L)$.
 This means that when $d\ge 2$, any local minimizer of $f_L(L)$ is
nonisolated and has a singular Hessian matrix.
	
	Next, we consider the correspondence between the local minimizers of \cref{eq:VL} and its rotation-reduced formulation \cref{eq:VLtriang}.
	
	\begin{theorem}\label{th:L*l}
		The following statements hold. 
		\begin{enumerate}
			\item
			If  $L_*$ is a local minimizer of $f_L$, then any $\ell_*$ satisfying \begin{equation}\label{eq:Lell}
				L_* = \LTriag(\ell_*)Q^T,
\end{equation}
for some $Q \in \cO$, is a local minimizer of $f_{\ell}$.
			\item
\label{item:fullcolrankL}
			\tred{
				If $\ell_*$ is a local minimizer of $f_{\ell}$, and the first $d$ rows of $\LTriag( \ell_*)$ are linearly independent, 
			then $L_* = \LTriag( \ell_*)$ is a local minimizer of $f_L$.}
		\end{enumerate}
	\end{theorem}
	\begin{proof}
		\begin{enumerate}
			\item
			Suppose $L_*$ is a local minimizer of $f_L$, meaning there exists $r>0$ such that
			\begin{equation}\label{eq:lL}
				f_L(L)\ge f_L(L_*),\ \forall  L\ :\ \|L- L_*\|_F\le r.
			\end{equation}
For any $\ell\in\R^{t_\ell}$ satisfying $\|\ell-\ell_*\|\le r$,
we let $L = \LTriag(\ell) Q^T$ and we have
			$$\| L - L_* \|_F = \|\LTriag(\ell)-\LTriag ( \ell_*)\|_F=\|\ell-\ell_*\|\le r.$$
			Then, from \cref{eq:Lell} and \cref{eq:lL} 
			we have 
			\[f_\ell(\ell)=f_L(\LTriag(\ell)) = f_L(L) \ge f_L(L_*)=f_L(\LTriag( \ell_*))=f_\ell( \ell_*).\]
			Therefore, $\ell_*$ is a local minimizer of $f_{\ell}$.
			
			\item
			We prove \Cref{item:fullcolrankL} by
contradiction. Suppose \tred{$ L_*=\LTriag( \ell_*)$} is not a
local minimizer of $f_L$. Then there exists a sequence $L_k, k=1, 2,
\ldots$ such that
			\begin{equation}\label{eq:nlL}
				\lim_{k\rightarrow +\infty}L_k=L_*,\ f_L(L_k)< f_L( L_*).
			\end{equation}
Consider the QR decompositions of $L_k^T, k=1,
2, \ldots$, i.e.,~there exist $Q_k\in \cO, k=1, 2, \ldots$, and
upper triangular matrices $R_k, k=1, 2, \ldots$, such that
			\begin{equation}\label{QRLk}
				L_k^T=Q_kR_k, k=1, 2, \ldots.
			\end{equation}
			Since $\|Q_k\|_2 = 1$ for all $k=1, 2,
\ldots$, the bounded sequence $\{Q_k\}$ has a convergent
			subsequence. Without loss of generality, we directly assume that $\lim_{k\rightarrow +\infty}Q_{k}=Q_*$. 
According to \cref{eq:nlL} and \Cref{QRLk}, we have 
\begin{equation}\label{eq:inf}
R_*^{T}:=L_* Q_*=\lim_{k\rightarrow +\infty}L_kQ_k=\lim_{k\rightarrow +\infty}R_k^T. 
\end{equation}
			By \cref{eq:inf}, $R_*^{T}$ is a triangular matrix. Since the first $d$ rows of $L_*=\LTriag( \ell_*)$ are linearly independent, the QR factorization of $L_*^{T}$ is unique  except for signs in each dimension. Thus, $Q_*$ is a diagonal matrix with diagonal elements being $-1$ or $1$. Let $$\ell_k=\LTriag^*(R_k^TQ_*^{T}), k=1, 2, \ldots.$$ By \cref{eq:inf},
			we get $$\lim_{k\rightarrow +\infty}\ell_k=\ell_*.$$
			By \cref{eq:nlL}, we have
			\begin{align*}
			f_\ell(\ell_k)=f_L(R_k^T)=f_L(L_k) & < f_L( L_*) = f_L(R_*^{T} Q_*^{T}) \\ & =f_L(R_*^{T})=f_\ell( \ell_*).
			\end{align*}
			Thus, $\ell_*$ is not a local minimizer of $f_{\ell}$, a
			contradiction.

		\end{enumerate}
	\end{proof}
	
	The optimality conditions of  \cref{eq:F} and \cref{eq:VL}  also have
	an equivalence relationship. To this end, we first note that the
	directional derivatives of $f(P)$ are zero in any translation.
	\begin{lemma}\label{le:12}
		For any $v\in\R^d$, we have 
$$\langle f'(P), ev^T
				\rangle =0  {\rm\ and\ }
f''(P)(ev^T, ev^T)=0. $$
	\end{lemma}	
	\begin{proof}
		For $t\in\R$, we have 
		\begin{equation*}
			\begin{array}{rcl}
				f(P+tev^T)&=&f(P)+t\langle f'(P), ev^T
				\rangle +\frac{t^2}{2}\langle  f''(P)(ev^T), ev^T
				\rangle+o(t^2).
			\end{array}	
		\end{equation*}
		Since $f(P+tev^T)=f(P)$ holds for all $v\in\R^d$ and $t\in\R$, we get 
		$$\langle f'(P), ev^T
		\rangle =\langle  f''(P)(ev^T), ev^T
		\rangle=0.$$
Moreover, we claim that 
\begin{equation}\label{eq:fprpre}
	f''(P)(ev^T)=0.
\end{equation}
According to \cref{eq:2ndorderHess} and \cref{eq:dJ}, we have
\begin{eqnarray}
		f''(P)(ev^T)&=&4\left[J^* J \right]\kvec(ev^T) + 2\left[\kvec \cK^*F(P)\Mat \right]\kvec(ev^T).\nonumber
\end{eqnarray}
Since $\nul(\cK) = \range(S_e)$ and $\range(\cK^*) = \{S\in\Sn:\ Se=0\}$ in \Cref{item:cK,item:LaplKsNinvar} of
\Cref{lem:adjderivs}, we have
$$J(\kvec(ev^T))=\cK\left(\frac{Pve^T+ev^TP^T}{2}\right)=0,\
\cK^*F(P)(ev^T)=0.$$ Thus, \cref{eq:fprpre} holds.
\end{proof}

	\begin{theorem}\label{th:equi}For $P\in\R^{n\times d}$, denote $v=(P^T e)/n\in\R^d$, $P_v=P-ev^T$,  $L=V^TP_v$ where $V$ is defined in \cref{eq:Vmatr}, denote $L^T=QR$ where $R$ is upper triangular and $Q\in\R^{d\times d}$ is orthogonal. Then,
		the following are equivalent:
		\begin{itemize}
			\item[{\rm(i)}] the first (resp., second)-order necessary conditions of \cref{eq:F} hold at $P$; 
			\item[{\rm(ii)}] the first (resp., second)-order necessary conditions of \cref{eq:F} hold at $P_v$; 
			\item[{\rm(iii)}]  the first (resp., second)-order necessary conditions of \cref{eq:VL} hold at $L$; 
			\item[{\rm(iv)}]  the first (resp., second)-order necessary conditions of \cref{eq:VL} hold at $R^T$.
		\end{itemize}
	\end{theorem}
	\begin{proof}
	Since
		 \begin{eqnarray}
			&&f(P+t\Delta P)	=f(P)+t\langle f'(P), \Delta P \rangle +\frac{t^2}{2}\langle  f''(P)(\Delta P), \Delta P\rangle+o(t^2)\nonumber\\
			=&& f(P_v+t\Delta P)=f(P_v)+t\langle f'(P_v), \Delta P \rangle +\frac{t^2}{2}\langle  f''(P_v)(\Delta P), \Delta P\rangle+o(t^2)\nonumber
		\end{eqnarray}
	holds for all $\Delta P\in\R^{n\times d}$ and $t\in\R$, we have
		\begin{equation}\label{eq:f'Pv}
			f'(P_v)=f'(P),\ f''(P_v)=f''(P).
		\end{equation} 
By
\begin{eqnarray}
	&&f_L(L+t\Delta L)	=f_L(L)+t\langle f_L'(L), \Delta L \rangle +\frac{t^2}{2}\langle  f_L''(L)(\Delta L), \Delta L\rangle+o(t^2)\nonumber\\
	=&& f_L(R^T+t\Delta LQ)=f_L(R^T)+t\langle f_L'(R^T), \Delta LQ \rangle +\frac{t^2}{2}\langle  f_L''(R^T)(\Delta LQ), \Delta LQ\rangle+o(t^2),\nonumber
\end{eqnarray}
 we have
$$f_L'(R^T)=0\Leftrightarrow f_L'(L)=0,\ f_L''(R^T)\succeq
0\Leftrightarrow f_L''(L)\succeq 0.\footnote{Note that $f_L''(L)$ is a
positive semidefinite linear operator on $\R^{(n-1)\times d}$.}
$$
Thus, (i)$\Leftrightarrow$(ii) and (iii)$\Leftrightarrow$(iv). 

Now we prove (ii)$\Leftrightarrow$(iii). First, we prove the equivalence
of their first-order necessary conditions. According to \cref{eq:rev1}
		and $\range(\cK^*) =     
\SC$ (\Cref{lem:adjderivs},~\Cref{item:LaplKsNinvar}),
		we have
		\begin{equation}\label{eq:e'f'}
			e^Tf'(P)=2e^T\cK^*(F(P))P=0.
		\end{equation} 
By \Cref{pro:fL''}, \cref{eq:f'Pv}, \cref{eq:e'f'}, the definition of $V$, and \Cref{item:LaplKsNinvar} of \Cref{lem:adjderivs}, we obtain
 $$f'(P_v)=0\Longleftrightarrow f'_L(L)=V^Tf'(P_v)=0.$$
	Secondly, we prove the equivalence of their second-order necessary optimality conditions. 
		According to \cref{eq:2optfv0}, for any $\Delta
L\in\R^{(n-1)\times d}$, we have \begin{equation}\label{eq:fprprL}
			f''_L(L)(\Delta L, \Delta L)= V^Tf''(VL)V(\Delta L, \Delta L)=f''(VL)(V\Delta L, V\Delta L).
		\end{equation}
		According to \cref{eq:fprpre} in \Cref{le:12} and
\cref{eq:fprprL}, we have $f''_L(L)(\Delta L, \Delta L)\ge 0$ if, and
only if, 
		$$ f''(P_v)(\Delta P, \Delta P)=f''(P_v)(V\Delta L+ev^T, V\Delta L+ev^T)\ge 0.$$
(Note that we have proved a slightly stronger statement as the
semidefinite condition is treated separately from stationarity.)
	\end{proof}
	\begin{remark}
		The reduction from \cref{eq:VL} to \cref{eq:VLtriang}
may introduce additional stationary points. 
		Let $\LTriag(\ell_*)=R_*^{T}$. According to \cref{eq:1optfl},
		$f'_\ell(\ell)=0$ holds if, and only if, the lower
triangular part of $f'_L(R_*^{T})$  is zero. Moreover, to have a local minimizer correspondence, the assumption that the first $d$ rows of $R_*^{T}$ is
linear independent in \Cref{th:L*l},~\Cref{item:fullcolrankL} is needed.

	\end{remark}
	
	\section{Second-Order {Optimality} Conditions}
	\label{sect:optcondsens}
	In this section,
	we present the optimality conditions and derive a sufficient condition such that there is no \lngmp.
	First of all, the necessary and sufficient characterization for the 
	global minimizer is clear.
	\begin{lemma}\label{le:gOpt}
		A matrix $P \in \Rnd$ is a {\bf global minimizer} of \cref{eq:F} if, and only if, $\mathcal{D}(P)=\bar D$.
	\end{lemma}
	\begin{proof}
		Since $f(P)\ge0$ holds for all $P\in\Rnd$ and $f(\bar P)=0$, the
		global minimum of $f$ is $0$. By the definition of $f$ and
		property of norms , $f(P)=0$ holds if, and only if, $F(P)=\mathcal{D}(P)-\bar D=0$. 
	\end{proof}
In order to further characterize the second-order optimality conditions, we discuss essential properties of the matrices $H_1$ and $H_2$. 
	\begin{lemma}\label{le:H12succ}
		The matrix $H_1$ defined in \cref{eq:H12} 
		is always positive semidefinite. For $H_2$, the following holds:
		\begin{itemize}
		\item $H_2\succeq 0$ when  $F(P)$ is element-wise nonnegative,
		\item $H_2\preceq 0$ when $F(P)$ is element-wise nonpositive.
	\end{itemize} 
	\end{lemma}
	
	\begin{proof}
		For any $x\in\R^{nd}$, 
		$$x^TH_1x=\langle x, J^*Jx\rangle=\langle Jx, Jx\rangle\ge 0.$$
		Thus, $H_1$ is always positive semidefinite.
		By \Cref{lem:adjderivs},~\Cref{item:LaplKsNinvar}, 
		if $F(P)\ge(\le)\ 0$, then $\cK^*(F(P))\succeq(\preceq)\ 0$, which implies
		$$x^TH_2x=\langle x, \Mat^* \cK^*F(P)\Mat x\rangle=\langle \Mat x, \cK^*(F(P))\Mat x\rangle\ge(\le)\ 0$$
	for all $x\in\R^{nd}$. Thus, $H_2\succeq(\preceq)\ 0$  if $F(P)\ge(\le)\ 0$. 
	\end{proof}
	\begin{lemma}\label{le:H2}
		The matrix $H_2$ is the zero matrix if, and only if, $F(P)=0$ holds, which is equivalent to $P$ being a global minimizer of \cref{eq:F}.
	\end{lemma}
	\begin{proof}
		By \Cref{lem:adjderivs},~\Cref{item:LaplKsNinvar}, $\cK^*(S)=2(\Diag (Se)-S)$ and $\nul(\KK^*) = \Diag(\Rn)$. Since
		$\diag(F(P))=\diag(\mathcal{D}(P))-\diag(\bar D)=0$ is
always true, $\cK^*(F(P))=0$
		holds if, and only if, $F(P)=0$.
	\end{proof}

	\begin{lemma}\label{le:pz}
	{Let $\bar{P}$, with $\bar{P}^T e = 0$, be a global minimizer of \cref{eq:F}.}
		Suppose that $P$ is a stationary point for \cref{eq:F}
but is not a global optimizer. Then $H_2$ is not positive semidefinite. Specifically,
		{ $ \kvec(\bar P)^TH_2\kvec(\bar P)<0$.  }
	\end{lemma}
	\begin{proof}
		By \cref{eq:1opt}, we have $\langle P, \cK^*F(P)P\rangle$=0, and then
		\begin{equation*}
			\begin{array}{rcl}
				{  \kvec(\bar P)^TH_2\kvec(\bar P) }
				&=&\langle \bar P, \cK^*F(P) \bar P\rangle-\langle P, \cK^*F(P)P\rangle\\
				&=&\langle \cK(\bar P  \bar P^T), F(P)\rangle-\langle \cK(P P^T), F(P)\rangle\\
				&=&\langle \cK(\bar P  \bar P^T)- \cK(P P^T), F(P)\rangle\\
				&=&\langle \bar D- \cD(P), F(P)\rangle\\
				&=&-\langle F(P), F(P)\rangle\\
				&<&0.
			\end{array}
		\end{equation*}
		The last inequality holds because $P$ is not a global minimizer, which implies $F(P)\neq 0$ according to \Cref{le:gOpt}.
	\end{proof}
	Under the condition of \Cref{le:pz}, we have known that $$\langle \bar P, \cK^*F(P) \bar P\rangle<0,$$
	which implies that
	\begin{equation}\label{eq:cKF}
		\cK^*F(P)\nsucceq 0.
	\end{equation}
	
	We analyze the extreme case of $\bar D=0$.
	\begin{corollary}\label{cor:D0}
		If $\bar D=0$, then every stationary point is a global minimizer.
	\end{corollary}
	\begin{proof}
Suppose $P$ is a stationary point. As $\bar D=0$, we get $\bar p_1=\ldots=\bar p_n$ holds. 
		Since $\cK^*F(P)$ is a Laplacian (sum of its columns is zero), we have $\cK^*F(P)\bar P=0$.
		Combining this with the first-order condition~\cref{eq:1opt},
		we have 
		\begin{equation*}
			\begin{array}{rcl}
				{  \kvec(\bar P)^TH_2\kvec(\bar P) }
				&=&\langle \bar P, \cK^*(F(P)) \bar P\rangle-\langle P, \cK^*(F(P))P\rangle\\
				&=&0.
			\end{array}
		\end{equation*}
		From \Cref{le:pz}, we conclude that any stationary point $P$ for $f(P)$ is a global minimizer.
	\end{proof}

	Next, we consider another extreme case of $\bar D\neq 0, \cD(P)=0$.
	\begin{theorem}\label{th:4.6}
For $\bar D\neq 0$ and $P$ with $\cD(P)=0$ (i.e., all $p_i \equiv p$), $P$ is a stationary point and has nontrivial negative semidefinite Hessian:
\[
0\neq4H_1+2H_2\preceq 0.
\]
	\end{theorem}
	\begin{proof}
		Since $p_1=\ldots = p_n=p$ and $\cK^*(F(P))$ is a Laplacian, $P$ satisfies the first-order optimality condition \cref{eq:1opt}.
		Since $f(P)={\|\bar D- \cD(P)\|}_F^2={\|\bar D\|}_F^2>0$, $P$ is not a global minimizer. 
		By $$P\Delta P^T+\Delta PP^T=ep^T\Delta P^T+\Delta P pe^T=e(\Delta Pp)^T+(\Delta P p)e^T,$$
and $\nul(\cK) = \range(S_e)$ (\Cref{lem:adjderivs}, \Cref{item:cK}), $J=0$ defined in \cref{eq:dJ} holds, and then $H_1=0$. 
		Since $\cD(P)=0$ and $\bar D\ge 0$, $F(P)\le 0$ holds.
		According to \Cref{le:H12succ} and  \Cref{le:H2}, $0\neq
H_2\preceq 0$ holds. Therefore, the Hessian matrix satisfies
$0\neq4H_1+2H_2\preceq 0$. 
	\end{proof}
	
	\begin{remark}
If $P$ is a local {\bf
maximizer} of $f$, then necessarily $\cD(P)=0$ ($p_1=\cdots=p_n$). 
To see this, observe that $t = 1$ must locally maximize $g(t) = f(tP)$. 
		By the second-order necessary conditions, we have
$g'(1) = 0$ and $g''(1) \leq 0$. 
		Since $g'(t) = 4 t^3 \| \cD(P) \|_F^2 - 4 t \langle \bar D , \cD(P) \rangle$, 
		$0 = g'(1)$ implies that $\| \cD(P) \|_F^2 = \langle \bar D , \cD(P) \rangle$.  
		Then, $0 \geq g''(1) = 8 \| \cD(P) \|_F^2$ and we conclude that $\cD(P) = 0$.
	\end{remark}
	
	In the following, we present the condition under which there is no
	\lngm. Recall the equivalence between local minimizers of \cref{eq:VL}
	and \cref{eq:F} in \Cref{th:equi2}, we  analyze \cref{eq:VL} for convenience.
	\begin{theorem}\label{th:nd1}
		Any stationary point $L$ of \cref{eq:VL} satisfying ${\rm rank}(L)= n-1$ is a global minimizer.
	\end{theorem}
	\begin{proof}
		Since  $0 = f'_L(L) = 2V^T\cK^*F(VL)VL$, where the last equality follows from \cref{eq:1optfv0} and \cref{eq:rev1}, 
		the span of columns of $L$ is an $n-1$ dimensional
eigenvector space corresponding to the zero eigenvalue of the
$(n-1)\times (n-1)$ matrix $V^T\cK^*F(VL)V$.
Therefore $V^T\cK^*(F(P))V=0$. Combining this with $\range(\cK^*)=\SC$
from \Cref{lem:adjderivs},~\Cref{item:LaplKsNinvar}, we conclude that
$\cK^*(F(P))=0$, and then $H_2=0$. Thus, $L$ is a global minimizer according to \Cref{le:H2}. 
	\end{proof}
	As $L\in\R^{(n-1)\times d}$, the condition in \Cref{th:nd1} holds in the case that $d\ge n-1$ and $L$ is of full row rank. Next, we consider another case where $L$ is not full column rank. 
	\begin{theorem}\label{th:dimless}
		Suppose that $L$ is a non-globally-optimal stationary point of \cref{eq:VL} and
		\begin{equation}\label{eq:dimd}
			{\rank}(L)<d.
		\end{equation} 
		Then, the second-order necessary optimality conditions fail at $L$.
	\end{theorem}
	\begin{proof}
		Denote $P=VL$. According to \Cref{le:pz} and the subsequent discussion, \cref{eq:cKF} holds. Thus there exists $a\in\R^n$ such that $a^T\cK^*F(P)a<0$.
		Then, for any nonzero $w\in\R^d$,
		\begin{equation*}
			\begin{array}{rcl}
				\kvec(a w^T)^TH_2\kvec(a w^T)
				&=&\langle aw^T, \cK^*F(P)(a w^T)\rangle\\
				&=&\Tr(\cK^*F(P)(a w^T)(a w^T)^T) \\
				&=&w^Tw\Tr(\cK^*F(P)aa^T)\\
				&=&w^Tw a^T\cK^*F(P)a\\
				&<&0.
			\end{array}
		\end{equation*}
		
		By \cref{eq:dimd}, there exists a nonzero $w\in\R^d$ such that $w\in{\rm null}(L)$, meaning, 
		\begin{equation}\label{eq:pijv}
			Lw=0.
		\end{equation}
		We claim that $H_1\kvec(a w^T)=0$ holds.
		First,	we have
		$$J \kvec (a w^T)=\cK \cS VL \cT (a w^T).$$ 
		By \cref{eq:pijv}, we have $$L \cT (a w^T)=L\begin{bmatrix}
			a_1  w &a_2  w &\ldots & a_{n-1}  w 
		\end{bmatrix}=0.$$
		Thus, $H_1\kvec(a w^T)=0$ holds.
		In sum, we have 
		\[\kvec(a w^T)^T(4H_1+2H_2)\kvec(a w^T)<0,\]
		which implies the second-order necessary optimality condition
		\cref{eq:2opt} fails.
	\end{proof}
	Combining \Cref{th:nd1} and \Cref{th:dimless},  we present the main result of this section.

	\begin{theorem}
\label{thm:2ndorderopt}
		If $n\le d+1$, then any stationary point satisfying the second-order necessary optimality conditions is a global minimizer.
	\end{theorem}
	\begin{proof}
			Suppose that $n\le d +1$, and $L$ is a stationary point satisfying the second-order necessary optimality condition. 
			If $\rank(L)=n-1$, then $L$ is globally optimal
by \Cref{th:nd1}. If $\rank(L)<n-1$, then $\rank(L)< d$. If we assume $L$ is not a global minimizer, according to
\Cref{th:dimless}, $L$ does not satisfy the second-order necessary
optimality condition, a contradiction.
		\end{proof}
	
Recalling \Cref{rem:mainreducprobtriang}, we note that $n\le d+1$ is exactly the condition such that the underlying system of equations is square. 
When $n > d+1$ (overdetermined), it is possible to find
local nonglobal minimizers.

\section{\lngmp s Examples}

We now provide instances with \lngmp s. 
The data and codes are available at
\href{https://github.com/MengmengSong97/EDM-code}{https://github.com/MengmengSong97/EDM-code}.

We first provide an analytical
example with $d=1$ with a specific simple structure for the points in
$\Rd$, see~\Cref{sect:analytexlngm}. 
Then in~\Cref{sect:Kantd1,sect:Kantd2} we give two
examples where we numerically obtain \underline{approximate}
second-order stationary
points. Then, we analytically prove that the assumptions of the
Kantorovich theorem hold at these two points, i.e.,~this implies that
there exists \lngmp s in the neighborhoods.
We consider the sensitivity
analysis needed to analytically prove that our examples have \lngmp s.
We exploit the strength of the classical Kantorovich theorem for the
convergence of Newton's Method to exact stationary points when using the
approximate stationary points that we found as starting points. See 
\Cref{th:Kantorovich} and \Cref{th:Kantorovich2}, below.

{\subsection{{ An Explicit Example with a \lngm
with $d=1$}}
\label{sect:analytexlngm} 
We now present a simple explicit example with an \lngmp,
where we can analytically verify the \lngm. 
\begin{example}\label{ex:exactex}
	Let 
\[
d=1, n>6; \quad
\bar{P}^T=[\bar p_1, \cdots, \bar p_n],\ \tilde{P}^T=[\tilde p_1,
\cdots, \tilde p_n]\in \Rdn,
\]
 with 
\[
\bar p_1=2,\ \bar p_2=0,\ \bar p_3=\cdots=\bar p_n=1 \quad \text{and}
\quad \tilde p_1=\tilde p_2=0,\ \tilde p_3=\cdots=\tilde p_n=1.
\]
\end{example}
We now continue and illustrate that $\tilde P$ is a \lngm of~\cref{eq:F}
in~\Cref{prob:origF}.  Let \textdef{$E_{2,n-2}$} be the $2\times n-2$ matrix of
all ones, and $\begin{bmatrix}0 \end{bmatrix}$ denote the matrix of
zeros of appropriate size.
From the definitions of $F(\cdot)$, $\cK(\cdot)$ and $\cK^*(\cdot)$, we
have: 
\begin{eqnarray}
	F(\tilde P)&=&\cD(\tilde P)-\cD(\bar P)\nonumber\\
	&=&
\begin{bmatrix}
\begin{bmatrix}0 & 0\cr 0 & 0 \end{bmatrix}& E_{2,n-2} \\
		   E_{2,n-2}^T & \begin{bmatrix}0 \end{bmatrix}
\end{bmatrix}
-
\begin{bmatrix}
\begin{bmatrix}0 & 4\cr 4 & 0 \end{bmatrix}& E_{2,n-2} \\
		   E_{2,n-2}^T & \begin{bmatrix}0 \end{bmatrix}
\end{bmatrix}
\\&=&
\begin{bmatrix}
\begin{bmatrix}0 & -4\cr -4 & 0 \end{bmatrix}& 0_{2,n-2} \\
		   0_{2,n-2}^T & \begin{bmatrix}0 \end{bmatrix}
\end{bmatrix};
\end{eqnarray}
\[
	\cK^*(F(\tilde P))
	=2\begin{bmatrix}
	\begin{bmatrix}-4 & 4\cr 4 & -4\end{bmatrix} & 
	\begin{bmatrix}0 \end{bmatrix} \\
	\begin{bmatrix}0 \end{bmatrix} &
	\begin{bmatrix}0 \end{bmatrix}
	\end{bmatrix}.
\]
Moreover, we have
\begin{eqnarray}
&&\langle \cK(\tilde P\Delta P^T+\Delta P \tilde P^T), \cK(\tilde P\Delta	P^T+\Delta P \tilde P^T)\rangle
\nonumber\\
&=&4\sum_{i\neq j}[(\tilde p_i-\tilde p_j)(\Delta p_i-\Delta p_j)]^2
\nonumber\\
&=&8
\Delta P^T
\begin{bmatrix}
		\sum_{i=1}^n(\tilde p_1-\tilde p_i)^2 & -(\tilde p_1-\tilde p_2)^2 & \cdots & -(\tilde p_1-\tilde p_n)^2\\
		-(\tilde p_1-\tilde p_2)^2 & \sum_{i=1}^n(\tilde p_2-\tilde p_i)^2 & \cdots & -(\tilde p_2-\tilde p_n)^2\\
		\vdots& \vdots & \cdots & \vdots\\
		-(\tilde p_1-\tilde p_2)^2 & -(\tilde p_2-\tilde p_n)^2 & \cdots & \sum_{i=1}^n(\tilde p_n-\tilde p_i)^2
\end{bmatrix}
\Delta P\nonumber\\
&=&8
\Delta P^T
\begin{bmatrix}
(n-2)I & -E_{2,n-2} \cr
-E_{2,n-2}^T  & 2I
\end{bmatrix}
\Delta P.\nonumber
\end{eqnarray}
Then we can compute from \cref{eq:1storderHess,eq:fprpr} that the
derivative (gradient) is
\begin{equation}
f^\prime(\tilde P)=4[\Diag (F(\tilde P)e)-F(\tilde P)]\tilde P=0.
\end{equation}
And the Hessian quadratic form is
\begin{eqnarray}
	&&f^{\prime\prime}(\tilde P)(\Delta P, \Delta P)\nonumber\\
	&=& \langle \cK(\tilde P\Delta P^T+\Delta P \tilde P^T), \cK(\tilde P\Delta	P^T+\Delta P \tilde P^T)\rangle +2\langle F(\tilde P), \cK(\Delta P \Delta P^T)\rangle \nonumber
	\\&=&8
		\Delta P^T \left(
\begin{bmatrix}
(n-2)I & -E_{2,n-2} \cr
-E_{2,n-2}^T  & 2I
\end{bmatrix} + 
	\begin{bmatrix}
	\begin{bmatrix}-2 & 2\cr 2 & -2\end{bmatrix} & 
	\begin{bmatrix}0 \end{bmatrix} \\
	\begin{bmatrix}0 \end{bmatrix} &
	\begin{bmatrix}0 \end{bmatrix}
	\end{bmatrix}
	\right)
\Delta P.\nonumber
\nonumber
\end{eqnarray}

The Hessian $f^{\prime\prime}(\tilde P)=\nabla^2f(\tilde P)$  is a rank
$3$ update of $16 I$. It
is positive semidefinite with nullspace $\spanl(e)$ if, and only if,
$n\ge 7$. (It is indefinite when $n<6$.) 
Thus, $\tilde P$ is a second-order stationary point of the
problem with data given above.
Next, we prove that $\tilde P$ is a \lngm by considering the reduced
formulation \cref{eq:VL}.

To center as done in \Cref{th:equi2}, we let 
$$\tilde v=\frac{\tilde P^Te}{n}\in\R,\ \tilde P_*=\tilde P-e\tilde v^T,\ \tilde L=V^T\tilde P_*,$$
to obtain $\tilde L$.
According to \Cref{th:equi}, $\tilde L$ is a second-order stationary point of the function $f_L(L)$. 
By \Cref{pro:fL''}, we have $f''_L(\tilde L)=V^Tf''(V\tilde L)V=V^Tf''(\tilde P)V$. 
Since $f''(\tilde P)$ has a one-dimensional nullspace ($\spanl\{e\}$),
its restriction to $\spanl\{e\}^\perp$ is positive
definite. Thus, $f''_L(\tilde L)$ is positive definite. This implies
that $\tilde L$ satisfies the second-order sufficient optimality
conditions for a strict local minimum of $f_L(L)$.
By \Cref{th:equi2} and \Cref{le:gOpt}, $\tilde P$ is in fact a 
\lngm of $f(P)$. Note that
the objective function value $f(\tilde P) = 16>0 = f(\bar P)$,
thus confirming that $\tilde P$ is not a global minimum.
}

\subsection{{Examples via Kantorovich Theorem and Sensitivity Analysis}}
\label{sect:KantSens}
{We now present \Cref{ex:counterexample,ex:counterexample2},
with $d=1,2$, respectively, where we first find an approximate
second-order stationary point $\tilde L$ numerically that has a sufficiently large
(positive) objective value; and then we prove that 
there is a
\emph{\lngm nearby} using the Kantorovich theorem and sensitivity analysis.}

\subsubsection{Case $d=1$}
\label{sect:casedone}
In this case we analyze a configuration {matrix
$\tilde P=V\tilde L\in \R^{n\times 1}, 
\tilde L\in \R^{(n-1)\times 1}$}, satisfying:
\begin{equation}
\label{eq:configapproxlngm}
\begin{array}{rll}
\\ \text{Objective value}: & f(\tilde P)=f_L(\tilde L) >\tilde f_L,
 & \text{for some (large)  } \tilde f_L >0;
\\ \text{Near stationarity}: &  \|\nabla f_L(\tilde L)\| <\tilde g_L, &
\text{for
some (small)  } \tilde g_L> 0;
\\ \text{Local Convexity}: &  \lambda_{\min}(\nabla^2 f_L(\tilde L)) > \tilde
\lambda_L,  & \text{for some (large)  }  \tilde \lambda_L > 0.  
\end{array}
\end{equation}

While theoretically exact, the floating-point representations introduce round-off errors in finite-precision computations. 
Our MATLAB implementation performs complete finite-precision arithmetic analysis in order to compute the rigorous bounds $\tilde f_L$, $\tilde g_L$, and $\tilde \lambda_L$  (see \Cref{footn:papercodelink}).


		We apply the
		classical Kantorovich theorem,
		e.g.,~\cite[Thm 5.3.1]{DennSch:83}, to show that there
is a point \emph{nearby} that satisfies: (i) it is an \emph{exact}
 stationary point;
	(ii) the function value is positive; and (iii) the
		Hessian is still positive definite. This provides an
		analytic proof that we have a theoretically
verified \lngm near $\tilde L$.

\label{sect:Kantd1}
\begin{example}\label{ex:counterexample}
			An example with $n=50, d=1$ is given, with data
$\bar P=V\bar L,\ \tilde P=V\tilde L \in \Rnd$.
(See the footnote below.\footnote{\label{footn:papercodelink}
The data and codes are available at
\href{https://github.com/MengmengSong97/EDM-code}{https://github.com/MengmengSong97/EDM-code}})
Matrix $\bar D$ is the distance matrix obtained from $\bar L$ by
			$\bar D=\cK(V\bar L(V\bar L)^T)$.
			Thus,  $\bar L$ is a global minimizer. 
			$\tilde L$ is a numerically convergence point obtained by a trust region method with random
initialization, 
			where the objective value is
			\begin{equation}\label{eq:fL}
				f_L(\tilde L) >2.65\times10^{3},
			\end{equation}
			the absolute and relative gradient norms are 
			\begin{equation}\label{eq:gl}
				\|\nabla f_L(\tilde L)\|<10^{-7},\ \frac{\|\nabla f_L(\tilde L)\|}{1+f_L(\tilde L)} 	<3.53\times10^{-18},
			\end{equation}
			and the least eigenvalue of the Hessian matrix is 
			\begin{equation}\label{eq:lmin}
				2.12\times 10^2>\lambda_{\min} (\nabla^2 f_L(\tilde L) ) > 2.10\times 10^2.
			\end{equation}
		\end{example}
	In the following, we will verify that $f_L(L)$ has a \lngm.
		\begin{remark}
			The problem to find a \lngm is a \emph{nonlinear least squares} problem.
			The standard approach for nonlinear least squares
			is to use the Gauss-Newton method, which simplifies the Hessian $4H1 + 2H_2$ (see \cref{eq:2ndorderHess,eq:H12})
			by keeping only $4H_1$ 
			and omitting $2H_2$ 
			(same for $L$, cf.~\cref{eq:2optfv0}).
			This approximation relies on the assumption that
$f_L(L)$ is near zero, a condition we intentionally avoid, as this would yield the global minimum.
		\end{remark}
	
Now, we find an estimate for the Lipschitz constant $\gamma>0$ for the Hessian matrix of $f_L$. From our numerical output, we know that the smallest eigenvalue
		$\lambda_{\min}(\nabla^2 f_L(\tilde L))> 0$. By continuity of
		eigenvalues, we are guaranteed that this holds in a neighbourhood of
		$\tilde L$, which is now estimated in \Cref{prop:Lipconstant}.
		\begin{proposition} 
\label{prop:Lipconstant}
Let $r>0$ and $\tilde L\in \Rnod$ be given. If
			\begin{equation}\label{eq:gammal}
				\gamma\ge24\sqrt{2}\left(\sum_{i, j}\|(V\tilde L)[i,:]-(V\tilde
				L)[j,:]\|_F+2n^{3/2}r\right),
			\end{equation}
			then $\gamma$ is a Lipschitz constant for the Hessian of $f_L$ in the radius-$r$
			neighborhood of $\tilde L$, i.e.,
			\begin{equation}\label{eq:gamma}
				\|\nabla^2 f_L(\hat L) -\nabla^2 f_L(\check L)\|_2\le \gamma \|\hat L-\check L\|_F, \quad \text{for\ all  }  \hat L, \check L\in B_r(\tilde L).
			\end{equation}
			Moreover, 
			\begin{equation}\label{eq:gammas}\lambda_{\min}(\nabla^2 f_L (L)) \ge \lambda_{\min}(\nabla^2f_L(\tilde L))- \gamma r, \quad \text{for\ all  }  L \in B_r(\tilde L).	\end{equation}
		\end{proposition} 
		\begin{proof}
			By the definition of the induced norm, \cref{eq:gamma} is equivalent to \begin{equation}\label{eq:gamma2}
				|f_L''(\hat L)(\Delta L, \Delta L) -f_L''(\check L)(\Delta L, \Delta L)|\le \gamma \|\hat L-\check L\|,
			\end{equation}
for all  $\hat L, \check L\in B_r(\tilde L), \|\Delta L\|=1$.
			Let $$\hat P=V\hat  L, \check P=V\check L, \Delta P=V\Delta L, \tilde P = V \tilde L.$$
			According to \cref{eq:Hessiantwoterms} and \Cref{pro:fL''}, we have
\begin{equation*}
	\begin{array}{rcl}
f_L''(\hat L)(\Delta L, \Delta L)
	&=&
f''(\hat P)(\Delta P, \Delta P)
\\&=&
					\|\cK(\hat P\Delta P^T+\Delta P \hat P^T)\|_F^2+2\langle F(\hat P), \cK(\Delta P \Delta P^T)\rangle
\\&=&
		\sum_{i, j}(2\hat p_i^T\Delta p_i+2\hat p_j^T\Delta p_j-2\hat p_i^T\Delta p_j-2\hat p_j^T\Delta p_i)^2
\\&& \quad +2\sum_{i, j}\|\hat p_i-\hat p_j\|^2\|\Delta p_i-\Delta p_j\|^2
\\&=&
	4\sum_{i, j}[(\hat p_i-\hat p_j)^T(\Delta p_i-\Delta p_j)]^2+2\sum_{i, j}\|\hat p_i-\hat p_j\|^2\|\Delta p_i-\Delta p_j\|^2.
				\end{array}
			\end{equation*}
			The calculations about $\check L$ are similar, 
			implying 
			\begin{equation*}
				\begin{array}{ccl}
					&&f_L''(\hat L)(\Delta L, \Delta L)-f_L''(\check L)(\Delta L, \Delta L)\\
					=&&4\sum_{i, j}[(\hat p_i-\hat p_j)^T(\Delta p_i-\Delta p_j)]^2-[(\check p_i-\check p_j)^T(\Delta p_i-\Delta p_j)]^2\\
					&& \qquad  +2\sum_{i, j}(\|\hat p_i-\hat p_j\|^2-\|\check p_i-\check p_j\|^2)\|\Delta p_i-\Delta p_j\|^2\\
					=&&4\sum_{i, j}(\hat p_i-\hat p_j-\check p_i+\check p_j)^T(\Delta p_i-\Delta p_j)(\hat p_i-\hat p_j+\check p_i-\check p_j)^T(\Delta p_i-\Delta p_j)\\
					&& \qquad +2\sum_{i, j}(\hat p_i-\hat p_j-\check p_i+\check p_j)^T(\hat p_i-\hat p_j+\check p_i-\check p_j)\|\Delta p_i-\Delta p_j\|^2.
				\end{array}
			\end{equation*}
			Then, 
			\begin{equation*}
				\begin{array}{ccl}
					&&|f_L''(\hat L)(\Delta L, \Delta L)-f_L''(\check L)(\Delta L, \Delta L)|\\
					\le&&6\sum_{i, j}\|\hat p_i-\hat p_j-\check p_i+\check p_j\|\|\hat p_i-\hat p_j+\check p_i-\check p_j\|\|\Delta p_i-\Delta p_j\|^2.
				\end{array}
			\end{equation*}
			
			Since $\|\Delta P\|_F=\|V\Delta L\|_F=\|\Delta L\|_F=1,$
			\begin{equation*}
					\|\Delta p_i-\Delta p_j\|^2
					\le2(\|\Delta p_i\|^2+\|\Delta p_j\|^2)
					\le2.
			\end{equation*}
			From $\hat L, \check L\in B_r(\tilde L)$, it follows that $\hat P, \check P \in B_r(\tilde P)$.
	Then, 
			we have
\begin{equation*}
		\begin{array}{rcl}
					\|\hat p_i-\hat p_j-\check p_i+\check p_j\|
					\le&&\|\hat p_i-\check p_i\|+\|\hat p_j-\check p_j\|\\
					\le &&\sqrt{2}\sqrt{\|\hat p_i-\check p_i\|^2+\|\hat p_j-\check p_j\|^2}\\
	\le &&\sqrt{2}\|\hat L-\check L\|_F,
				\end{array}
			\end{equation*}where the first inequality follows from the triangle inequality, the second from the Cauchy-Schwarz inequality, and the third from the fact that $$\|\hat P-\check P \|_F=\|V\hat L-V\check L \|_F=\|\hat L-\check L \|_F.$$ We also have 
\begin{equation*}
	\begin{array}{rcl}
		\ && \sum_{i, j}\|\hat p_i-\hat p_j+\check p_i-\check p_j\|\\
					=&&\sum_{i, j}\|2(\tilde p_i-\tilde p_j)+(\hat p_i-\tilde p_i)-(\hat p_j-\tilde p_j)+(\check p_i-\tilde p_i)-(\check p_j-\tilde p_j)\|\\
	\le&&2\sum_{i, j}\|\tilde p_i-\tilde p_j\|+\sum_{i, j}\|\hat
p_i-\tilde p_i\|+\sum_{i, j}\|\hat p_j-\tilde p_j\|+\\
  && \qquad\sum_{i, j}\|\check p_i-\tilde p_i\|+\sum_{i, j}\|\check p_j-\tilde p_j\|\\
	=&&2\sum_{i, j}\|\tilde p_i-\tilde p_j\|+n\sum_{i}\|\hat
p_i-\tilde p_i\|+n\sum_{j}\|\hat p_j-\tilde p_j\|+\\
&& \qquad \quad n\sum_{i}\|\check p_i-\tilde p_i\|+n\sum_{ j}\|\check p_j-\tilde p_j\|\\
		\le&&2\sum_{i, j}\|\tilde p_i-\tilde p_j\|+4n^{3/2}r,
				\end{array}
			\end{equation*}
where the last inequality follows from the H{\"o}lder inequality.
	Thus, 
	\begin{equation*}
		\begin{array}{rcl}
					|f_L''(\hat L)(\Delta L, \Delta L)-f_L''(\check L)(\Delta L, \Delta L)|
					\le&&12\sqrt{2}\|\hat L-\check L \|_F\sum_{i, j}\|\hat p_i-\hat p_j+\check p_i-\check p_j\|\\
\le&&24\sqrt{2}\|\hat L-\check L \|_F\left(\sum_{i, j}\|\tilde
	p_i-\tilde p_j\|+2n^{3/2}r\right),
				\end{array}
			\end{equation*}
			{applying \cref{eq:gammal} turns out \cref{eq:gamma}.} 
			By \cref{eq:gamma2}, we have  \begin{equation*}
				\begin{array}{rcl}
					f_L''(L)(\Delta L, \Delta L)
					=&&f_L''(\tilde L)(\Delta L, \Delta L)-(f_L''(\tilde L)(\Delta L, \Delta L)-f_L''(L)(\Delta L, \Delta L) )\\
					\ge&&f_L''(\tilde L)(\Delta L, \Delta L)-|f_L''(L)(\Delta L, \Delta L) -f_L''(\tilde L)(\Delta L, \Delta L)|\\
					\ge&& \lambda_{\min}(\nabla^2f_L(\tilde L))- \gamma \|\hat L-\tilde L\|_F\\
					\ge&& \lambda_{\min}(\nabla^2f_L(\tilde L))- \gamma r,
\quad \text{for\ all  }  L\in B_r(\tilde L), \|\Delta L\|_F=1.
				\end{array}
			\end{equation*}
			Thus, we obtain \cref{eq:gammas}.
		\end{proof}

		To verify the existence of a \lngm  for \Cref{ex:counterexample}, we calculate the Lipschitz constant estimated in
		\Cref{prop:Lipconstant}.
		Let $r=10^{-3}$. 
		Since $$\sum_{i,j}\|(V\tilde L)[i,:]-(V\tilde L)[j,:]\|<2.13\times 10^3,$$ \cref{eq:gammal} gives
		$$\gamma=7.24\times 10^4.$$
		Moreover, by 
		\cref{eq:gammas}, we have
		\begin{equation}\label{eq:gammas2}\lambda_{\min}(\nabla^2 f_L (L)) \ge 211- 7.24\times 10^4\times r=138.6>0, 
			\quad \text{for\ all  } L \in B_r(\tilde L).
		\end{equation}
		That is, we find a neighbourhood where the Hessian stays positive semidefinite.
		Next, we prove that the
		objective stays sufficiently positive in a region around $\tilde L$.

		\begin{lemma}
			\label{lem:posobjf}
			Let the configuration $\tilde P = V\tilde L\in \Rnod, \tilde L\in\Rnod$
			and positive parameters
			$\bar f_L, r \in \Rpp$, be given. 
			Suppose that $f_L(\tilde L) > \bar f_L$ and that
			the Hessian $\nabla^2 f_L$ is \textdef{uniformly
positive definite} in the $r$-ball
			around $\tilde L$, i.e.,
			\begin{equation}
				\label{eq:uniformposdef}
				\lambda_{\min}(\nabla^2 f_L(L)) > 0, \quad \text{for\ all  }  L \in B_r(\tilde L).
			\end{equation}
			Then $f_L$ is \textdef{positively uniformly
bounded} from below in $B_r(\tilde L)$, i.e.,
			\[
			f_L(L) > \bar f_L > 0, 
			\quad \text{for\ all  } \| L-\tilde L\|_F\leq 
			\min\left\{r,\frac {f_L(\tilde L)-\bar f_L}{\|\nabla f_L(\tilde L)\|_F}\right\}.
			\]
		\end{lemma}
		\begin{proof}
By the positive definiteness assumption of the Hessian in the $r$-ball
			$B_r(\tilde L)$, we can apply convexity of $f_L$
in the ball. Therefore,
for all {$L\in\Rnod$ such that $\| L-\tilde L\|_F\leq 
		\min\left\{r, (f_L(\tilde L)-\bar f_L)/\|\nabla
f_L(\tilde L)\|_F\right\}$}, we have
			\[
			\begin{array}{rcll}
				f_L(L) 
				& \geq &
				f_L(\tilde L) + 
				\langle \nabla f_L(\tilde L), L-\tilde L\rangle
				\\& \geq &
				f_L(\tilde L) - \|  \nabla f_L(\tilde L)\|_F \|L-\tilde L\|_F
				\\& > & \bar f_L > 0.
			\end{array}
			\]
		\end{proof}
		
		\index{$B_\delta(\tilde L)$, $\delta$-ball about $\tilde L$}
		\index{$\delta$-ball about $\tilde L$, $B_\delta(\tilde L)$}
		
		By \cref{eq:fL}, \cref{eq:gl} and considering $\bar f_L=10^{3}$, we get 
		$$
			\frac{f_L(\tilde L) - \bar f_L}{\|\nabla f_L(\tilde L)\|}> \frac{2.65\times10^{3}- 10^3}{10^{-7}} > r.$$
	According to \Cref{lem:posobjf} { with
\cref{eq:gammas2}} we conclude that
		\begin{equation}\label{eq:fL2}
			f_L(L) > \bar f_L > 0, \quad \text{for\ all  }L \in B_r(\tilde L).
		\end{equation}

		We now apply the classical Kantorovich theorem to prove the existence of a 
		unique \lngm point within a certain neighborhood. We reword
		the version in~\cite[Thm 5.3.1]{DennSch:83}.
		
		\begin{theorem}\label{th:Kantorovich}
			Let the configuration matrix $\tilde P=V\tilde L\in \Rnd, \tilde L\in \Rnod$ 
			be given. 
			Let $r\in \Rpp$ be found such that 
			$$\nabla^2 f_L(L)\succ 0, \quad \text{for\ all  }  L \in B_r(\tilde L),$$
			and $\bar f_L$ satisfying $$f_L(L) > \bar f_L > 0, \quad \text{for\ all  }  L \in B_r(\tilde L).$$
			Let $\gamma$ be a Lipschitz constant for the Hessian of $f_L$ in the
			$r$-ball about $\tilde L$. Set 
			\[
			\beta := \|\nabla^2 f_L(\tilde L)^{-1}\|_2  \quad \text{and} \quad 
			\eta :=  \|\nabla^2 f_L(\tilde L)^{-1}\nabla f_L(\tilde L)\|.
			\]
			Define $\gamma_R = \beta \gamma$ and $\alpha = \gamma_R\eta$.
			If $\alpha \leq \frac 12$ and $r\geq 
			r_0:= \frac {1-\sqrt{1-2\alpha}}{\beta\gamma}$, then the sequence
			$L_0=\tilde L, L_1, L_2, \ldots$, produced by
			\[
			L_{k+1} = L_k - \nabla^2f_{L}(L_k)^{-1} \nabla f_L(L_k), \, k = 0,1,\ldots
			\]
			is well defined and converges to $L_*$, a unique root of 
			the gradient $\nabla f_L$ in
			the closure of $B_{r_0}(\tilde L)$. If $\alpha <\frac 12$, then $L_*$ is 
			the unique zero of $\nabla f_L$ in
		 $B_{r_1}(\tilde L)$, where
			$$r_1:= \min\left\{r,\frac {1+\sqrt{1-2\alpha}}{\beta\gamma}\right\},$$ and
			$
			\|L_k-L_*\|_F \leq (2\alpha)^{2k} \frac \eta \alpha, \quad k = 0,1,\ldots.
			$
			Moreover, $L_*$ is a \lngm.
		\end{theorem}
		\begin{proof}
			The proof is a direct application of the Kantorovich theorem,
			e.g.,~\cite[Thm 5.3.1]{DennSch:83}, along with
			the above lemmas and corollaries in this section.
		\end{proof}
		
		As mentioned previously in this section, the conditions required in \Cref{lem:posobjf} and \Cref{th:Kantorovich} are fulfilled
		with
\[r=10^{-3},\ \gamma=7.24\times 10^4,\ \bar f_L=10^3.\]
		Plugging them and \cref{eq:fL}, \cref{eq:gl}, \cref{eq:lmin} into \Cref{th:Kantorovich}, we have
		\begin{equation*}
			\begin{array}{ccl}
				&&1/212<\beta<1/210,\\
				&& \eta\le\|\nabla^2 f_L(\tilde L)^{-1}\|_2\|\nabla f_L(\tilde L)\|\le1/210\times10^{-7},\\
				&&\gamma_R=\beta \gamma<1/210\times7.24\times 10^4,\\
				&& \alpha=\gamma_R\eta<(1/210\times7.24\times 10^4)\times(1/210\times10^{-7})<1/2,\\
				&&r_0=\frac{(1-\sqrt{1-2\alpha})}{\beta\gamma}<\frac{1-\sqrt{1-2\times1/210^2\times7.24\times10^{-3}}}{1/212\times7.24\times 10^4}<r.
			\end{array}
		\end{equation*}
Combining these with  \cref{eq:gammas2} and \cref{eq:fL2} via \Cref{th:Kantorovich}, we conclude that $f_L$ has a \lngm in $B_{r}(\tilde{L})$.

In \Cref{Ex1} we plot the known centered global minimizer $\bar P$ and
the centered numerical found $\tilde P$ near which it has been proved
there exists a \lngm. It is interesting to observe that the $\tilde p_i\approx \bar p_i$ for all $i\neq i_0$, 
all except the one $\bar p_{i_0}$ with the biggest absolute value, i.e., $|\bar p_{i_0}|>|\bar p_{i}|$, for all $i\neq i_0$; 
the corresponding $\tilde p_{i_0}$ has the biggest absolute
value among all $\tilde  p_{i}$ for $i = 1, \dots, n=50$, and has an opposite sign to $\bar p_{i_0}$. 
\begin{figure}[H]
\begin{center}
\includegraphics[width=1.0\textwidth]{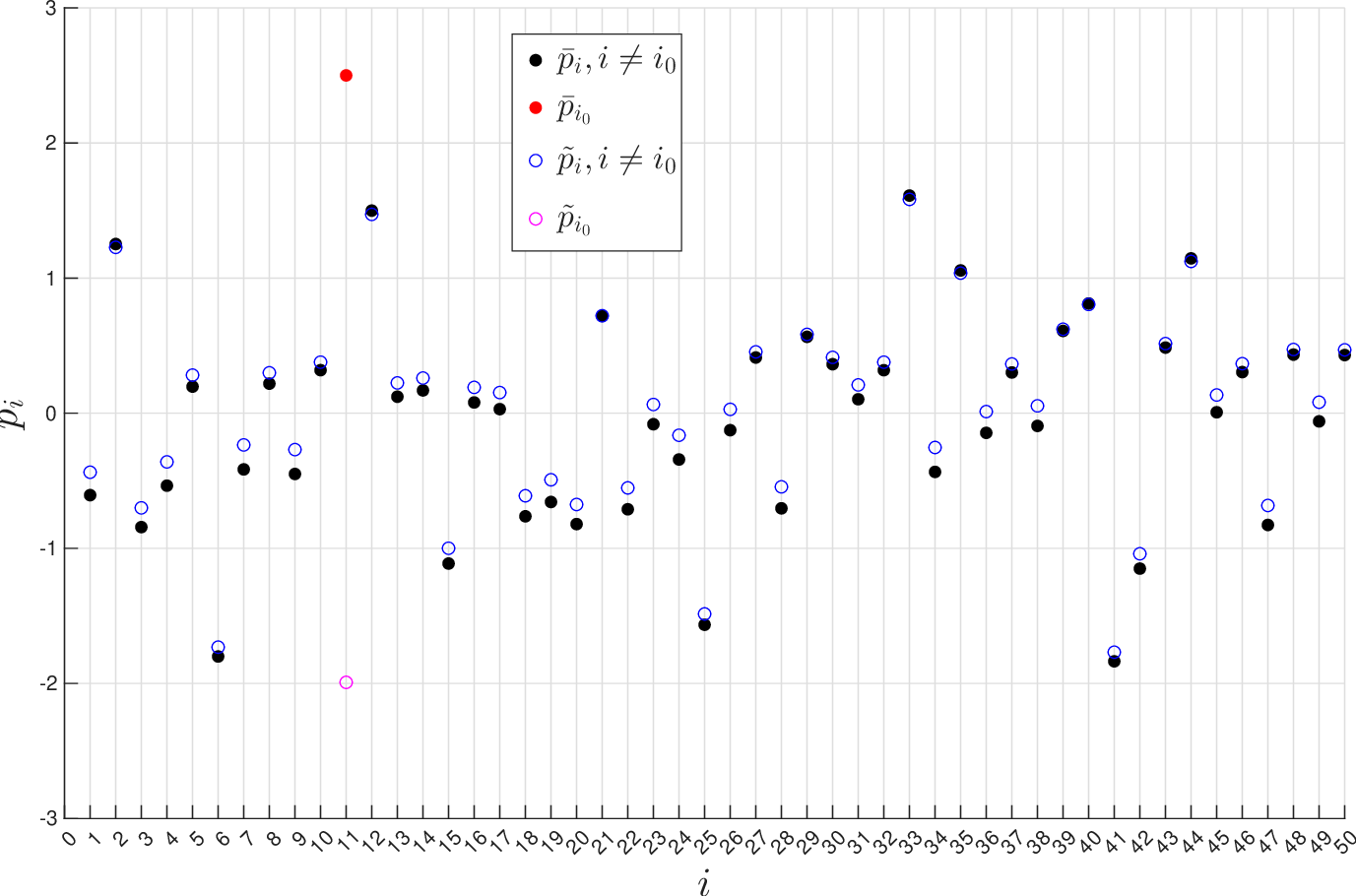}
\caption{values of \underline{global min} and
\underline{numerical near \lngmp}:
$\bar P, \tilde P\in\R^{50\times 1}$; in \Cref{ex:counterexample}.}
\label{Ex1}
\end{center}\end{figure}

Based on these observations, we generated examples with $d=1$ and
$n=100$. We randomly generated $\bar L$, then get centered $\bar P$ by
setting $\bar P=V\bar L$. Next, we find $\bar p_{i_0}$, which has the
biggest absolute value among all $\bar p_{i}$ for $i = 1, \dots, n$. We
define the starting point $\hat P$ by setting  $\hat p_i=\bar p_i$ for
$i\neq i_0$ and $\hat p_{i_0}=-\bar p_{i_0}$. Starting from $\hat L = V^T \hat P$,
the trust region method converges to a nonglobal second-order stationary
point of $f_L(L)$ with high frequency.

\subsubsection{Case $d=2$}
\label{sect:Kantd2}

As mentioned in \Cref{sect:mainprob} for the case of $d>1$, all local
minimizers of $f_L(L)$ are nonisolated, implying that the Hessian matrix, at any local minimizer of $f_L(L)$, is singular. 
We consider the model $f_\ell(\ell)$ for an example with $d=2$, and we analyze a configuration
$\tilde \ell\in \R^{t_\ell}$ satisfying equivalent conditions to
\cref{eq:configapproxlngm} but for $f_\ell$:
\begin{equation}
\label{eq:configapproxlngmell}
\begin{array}{rll}
\\ \text{Objective value}: & f_\ell(\tilde \ell)=f_\ell(\tilde \ell)
>\tilde f_\ell,
 & \text{for some (large)  } \tilde f_\ell >0;
\\ \text{Near stationarity}: &  \|\nabla f_\ell(\tilde \ell)\| <\tilde
g_\ell, &
\text{for
some (small)  } \tilde g_\ell> 0;
\\ \text{Local Convexity}: &  \lambda_{\min}(\nabla^2 f_\ell(\tilde \ell)) > \tilde
\lambda_\ell,  & \text{for some (large)  }  \tilde \lambda_\ell > 0.  
\end{array}
\end{equation}

We use Kantorovich Theorem to verify that this example has a strict local nonglobal minimizer $\ell_*$, where the first $2$ rows of $\LTriag( \ell_*)$ are linearly independent. According to \tred{\Cref{th:L*l,th:equi2}}, $\LTriag( \ell_*)$ is a local nonglobal minimizer of $f_L(L)$, and $V\LTriag( \ell_*)$ is a local nonglobal minimizer of  $f(P)$.
		
		\begin{example}\label{ex:counterexample2}
			An example with $n=100, d=2$ is given, with data
$\bar \ell, \tilde \ell\in\R^{197}$ presented,
see \Cref{footn:papercodelink}. 
 Matrix $\bar D$ is the distance matrix obtained from $\bar \ell$ by
			$$\bar D=\cK(V\LTriag(\bar \ell)\LTriag(\bar \ell)^TV^T).$$
			Thus, $\bar \ell$ is a global minimizer; 
			$\tilde \ell$ is a numerically convergence point obtained by a trust region method with random
initialization.
			The objective value is 
			\begin{equation}\label{eq:fell}
				f_\ell(\tilde \ell) > 9.99\times 10^3 ,
			\end{equation}
			the absolute and relative gradient norms  are 
			\begin{equation}\label{eq:gl2}
				\|\nabla f_\ell(\tilde \ell)\|<9.23\times 10^{-8},\ \frac{\|\nabla f_\ell(\tilde \ell)\|}{1+f_\ell(\tilde \ell)} 	< 9.24\times 10^{-12},
			\end{equation}
			and the least eigenvalue of the Hessian matrix is 
			\begin{equation}\label{eq:Hex2}8.58>\lambda_{\min} (\nabla^2 f_\ell(\tilde \ell) )>7.93.\end{equation}
		\end{example}
In the following \Cref{prop:Lipconstant2}, we verify that there exists a 
\lngm of $f_\ell(\ell)$. 

\begin{cor} 
			\label{prop:Lipconstant2}
			Let $r>0, \tilde \ell\in \Rnod$ be given. If
			\begin{equation}\label{eq:cal_gamma}
				\gamma\ge24\sqrt{2}\left(\sum_{i,
j}\|(V\LTriag(\tilde \ell))[i,:]-(V\LTriag(\tilde
\ell))[j,:]\|_F+2n^{3/2}r\right),
			\end{equation}
			then $\gamma$ is a Lipschitz constant for the Hessian of $f_\ell$ in the radius-$r$
			neighborhood of $\tilde \ell$:
			\begin{equation*}
				\|\nabla^2 f_\ell(\hat \ell) -\nabla^2 f_\ell(\check \ell)\|_2\le \gamma \|\hat \ell-\check \ell\|, \quad \text{for\ all  }   \hat \ell, \check \ell\in B_r(\tilde\ell).
			\end{equation*}
			Moreover, 
			\begin{equation*}
				\lambda_{\min}(\nabla^2 f_\ell (\ell)) \ge \lambda_{\min}(\nabla^2f_\ell(\tilde \ell))- \gamma r, \quad \text{for\ all  }  \ell \in B_r(\tilde \ell).	\end{equation*}
		\end{cor} 
\begin{proof}
The results follow from the fact that,
according to \cref{eq:2optfl}, 
the Hessian matrix of $f_\ell$ at $\tilde\ell$
is a submatrix of the Hessian matrix of $f_L$  at $\LTriag(\tilde\ell)$.
The steps are similar to the proof of \Cref{prop:Lipconstant}.
\end{proof}

In \Cref{ex:counterexample2}, we have $$\sum_{i, j}\|(V\LTriag(\tilde \ell))[i,:]-(V\LTriag(\tilde \ell))[j,:]\|<1.803\times10^4.$$ Let $r=10^{-5}$. Then $\gamma=6.12\times10^5$ satisfies \cref{eq:cal_gamma}. According to \Cref{prop:Lipconstant2} and \cref{eq:Hex2}, we have 
\begin{equation}\label{eq:lambdamin}\lambda_{\min}(\nabla^2 f_\ell (\ell)) \ge \lambda_{\min}(\nabla^2f_\ell(\tilde \ell))- \gamma r>0, \quad \text{for\ all  }   \ell \in B_r(\tilde \ell).\end{equation}

We now continue to extend the results from~\Cref{sect:casedone} to this
$d=2$ case.
	\begin{cor}\label{le:fellb}
			Suppose that $f_\ell(\tilde \ell) > \bar f_\ell$ and that
			the Hessian $\nabla^2 f_\ell$ is uniformly positive definite in the $r$-ball
			around $\tilde \ell$: 
			\begin{equation*}
				\lambda_{\min}(\nabla^2 f_\ell(\ell)) > 0, \quad \text{for\ all  } \ell \in B_r(\tilde \ell).
			\end{equation*}
			Then, $f_\ell$ is positively uniformly bounded below in $B_r(\tilde \ell)$: 
			\[
			f_\ell(\ell) > \bar f_\ell > 0, 
			\quad \text{for\ all  } \| \ell-\tilde \ell\|\leq 
			\min\left\{r,\frac {f_\ell(\tilde \ell)-\bar f_\ell}{\|\nabla f_\ell(\tilde \ell)\|}\right\}.
			\]
		\end{cor}
		\begin{proof}
The results follow as in~\Cref{lem:posobjf}.
		\end{proof}

Let $\bar f_L=10^3$. By \cref{eq:fell} and \cref{eq:gl2}, we have $$\frac {f_\ell(\tilde \ell)-\bar f_\ell}{\|\nabla f_\ell(\tilde \ell)\|}>\frac{8.99\times 10^3}{9.23\times 10^{-8}}>r.$$ 
Thus, according to \Cref{le:fellb}, we have 
	\begin{equation}\label{eq:ell2}
		f_\ell(\ell) > \bar f_\ell > 0, \quad \text{for\ all  }   \ell \in B_r(\tilde \ell).
	\end{equation}	
		\begin{cor}\label{th:Kantorovich2}
			Let $\tilde \ell\in \R^{t_\ell}$ 
			be given and $r\in \Rpp$ be found such that 
			\begin{equation}\label{eq:Hpositive2}\nabla^2 f_\ell(\ell)\succ 0, \quad \text{for\ all  } \ell \in B_r(\tilde \ell),
			\end{equation} 
			and $\bar f_\ell$ satisfy 	$$f_\ell(\ell) > \bar f_\ell > 0, \quad \text{for\ all  }  \ell \in B_r(\tilde \ell).$$
			Let $\gamma$ be a Lipschitz constant for the Hessian of $f_\ell$ in the
			$r$-ball about $\tilde \ell$. Set 
			\[
			\beta := \|\nabla^2 f_\ell(\tilde \ell)^{-1}\|_2, \quad \text{and} \quad 
			\eta :=  \|\nabla^2 f_\ell(\tilde \ell)^{-1}\nabla f_\ell(\tilde \ell)\|.
			\]
			Define $\gamma_R = \beta \gamma$ and $\alpha = \gamma_R\eta$.
			If $\alpha \leq \frac 12$ and $r\geq 
			r_0:= \frac {1-\sqrt{1-2\alpha}}{\beta\gamma}$, then the sequence
			$\ell_0=\tilde \ell, \ell_1, \ell_2, \ldots$, produced by
			\[
			\ell_{k+1} = \ell_k - \nabla^2f_{\ell}(\ell_k)^{-1} \nabla f_\ell(\ell_k), \, k = 0,1,\ldots,
			\]
			is well defined and converges to $\ell_*$, a unique root of 
			the gradient $\nabla f_\ell$ in
			the closure of $B_{r_0}(\tilde \ell)$. If $\alpha <\frac 12$, then $\ell_*$ is 
			the unique zero of $\nabla f_\ell$ in
			the closure of $B_{r_1}(\tilde \ell)$, 
			$$r_1:= \min\left\{r,\frac {1+\sqrt{1-2\alpha}}{\beta\gamma}\right\}$$ and
			\[
			\|\ell_k-\ell_*\| \leq (2\alpha)^{2k} \frac \eta \alpha, \quad k = 0,1,\ldots.
			\]
			Moreover, $\ell_*$ is a \lngm.
		\end{cor}
		\begin{proof}
As in \Cref{th:Kantorovich}, the proof is a direct application of the
Kantorovich theorem.
		\end{proof}
Plugging \cref{eq:fell}, \cref{eq:gl2},  \cref{eq:Hex2}, and 
\[r=10^{-5},\ \gamma=6.12\times10^5,\ \bar f_L=10^3,
\]
into \Cref{th:Kantorovich2}, we have
		\begin{equation*}
			\begin{array}{ccl}
				&&1/7.93>\beta>1/8.58,\\
				&& \eta\le\|\nabla^2 f_\ell(\tilde \ell)^{-1}\|_2\|\nabla f_\ell(\tilde \ell)\|<1/ 7.93\times 9.23\times10^{-8},\\
				&&\gamma_R=\beta\gamma<1/7.93\times6.12\times10^5,\\
				&& \alpha=\gamma_R\eta<(1/7.93\times6.12\times10^5)\times(1/ 7.93\times 9.23\times10^{-8})<1/2,\\
				&&r_0=\frac{1-\sqrt{1-2\alpha}}{\beta\gamma}<\frac{1-\sqrt{1-2\times1/7.93^2\times6.12\times 9.23\times10^{-3}}}{1/8.58\times6.12\times10^5}<r.
			\end{array}
		\end{equation*}
Combining this with \cref{eq:lambdamin} and \cref{eq:ell2}, we conclude
from \Cref{th:Kantorovich2} that there exists a \lngm in $B_{r}(\tilde\ell)$.

{In \Cref{Ex2} we plot: the points $\bar p_i\in \mathbb{R}^2,\
i=1, \dots, n=100$, of the global configuration $\bar P\in \mathbb{R}^{100\times
2}$, and the corresponding points $\tilde p_i\in \mathbb{R}^2,\ i=1, \dots,
n$, of the numerical configuration $\tilde P\in \mathbb{R}^{100\times 2}$
near a proven \lngmp.
We note that $\bar p_i\approx \tilde p_i, \forall i=1, \dots, n$,
except for two indices ${i_0}$ and ${i_1}$; while $\tilde p_{i_0}$ and
$\tilde p_{i_1}$ appear to be the reflections of $\bar p_{i_0}$ and $\bar p_{i_1}$.  
This interesting observation is similar to what happens in
\Cref{ex:counterexample} as seen in \Cref{Ex1}.} 
\begin{figure}[H]
	\begin{center}
		\includegraphics[width=1.0\textwidth]{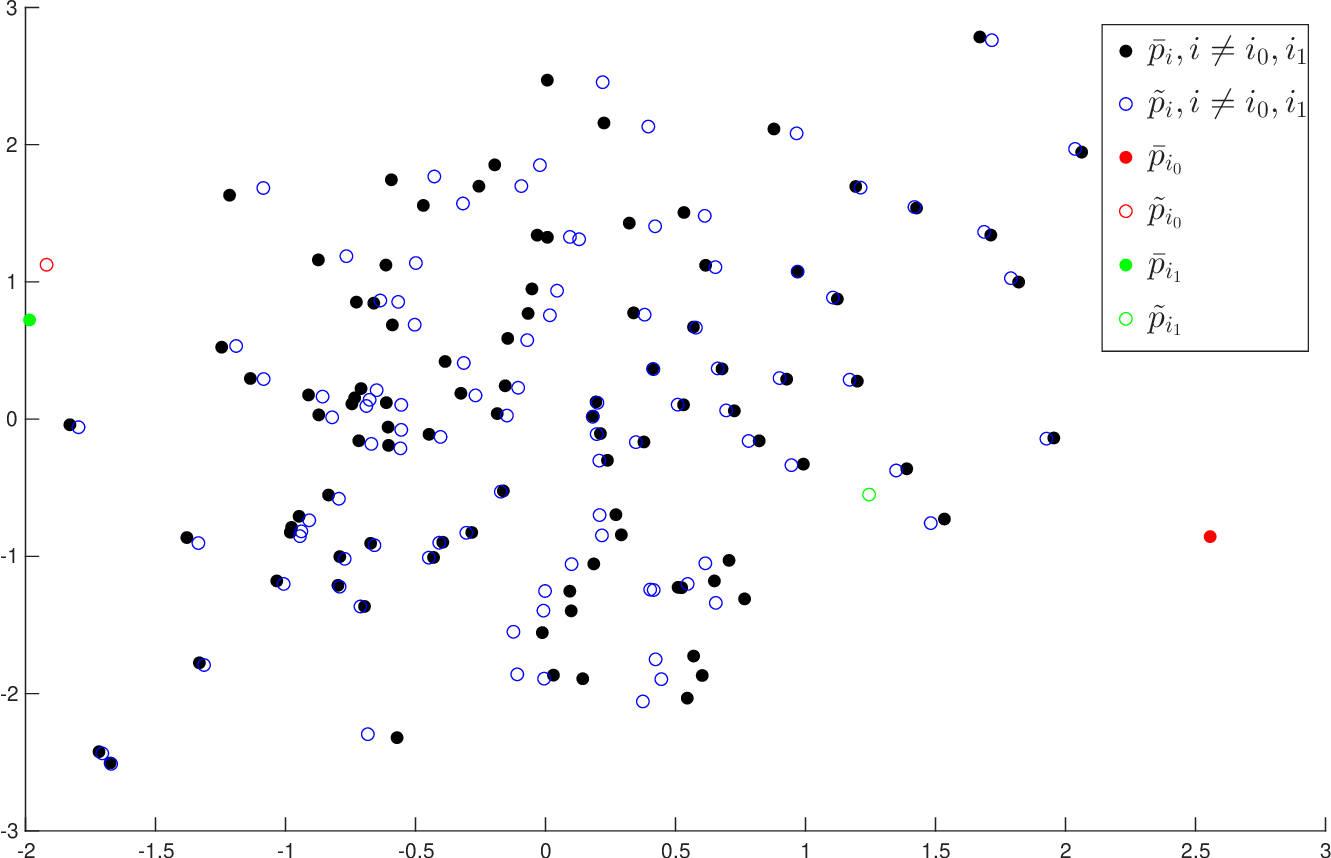}
\caption{coordinates of \underline{global min} and
\underline{numerical near proven
\lngmp}, resp.: 
$\bar P, \tilde P\in\R^{100\times 2}$; in \Cref{ex:counterexample2}.}
	\label{Ex2}
\end{center}\end{figure}
\section{Conclusion}\label{sec:final}
		In this paper, we addressed the nonconvex optimization problem arising from the exact recovery of points from a given \EDMp. 
		Our investigation led to significant advancements in
understanding the conditions under which the smooth stress function (as
known in the MDS literature) in \EDM problems has a \lngm. We established
that for the smooth stress function, which is a quartic in $P \in
\mathbb{R}^{n \times d}$, all second-order stationary points are global
minimizers when  $ n \leq d + 1 $. For $ n > d + 1 $, we not only
identified \lngm through numerical methods, but also used
Kantorovich's theorem to provide rigorous
analytical proofs confirming their existence.
{Moreover, based on the special patterns in those \lngmp s, we were able to build the analytical \Cref{ex:exactex}.}
	Our methodology includes two reduction techniques based on translation and rotation invariance. These reductions are necessary for the application of Kantorovich's theorem. 
		
	The findings of this research resolve a longstanding open question regarding the existence of \lngmp s in the context of MDS. Additionally, our research highlights the importance of second-order methods for minimizing the smooth stress function.
		\label{pape:conclusion}

{For the future we plan to explore the possibility of further
characterizing the properties of \lngmp s, e.g.,~with respect to
embedding dimensions and ranks.}
Moreover, our goal is to study conditions for the existence of {\bf
lngm} when $\bar D$ has inexact and missing entries, as this is closer to
the so called \EDM Completion Problem, 
e.g.,~\cite{KrislockWolk:10,MR3246296}. This work is a
first step in this direction as here we assume $\bar D$ is complete and a
true \EDM. 

\section*{Acknowledgments}
We thank the anonymous referees for the careful reading and meaningful
suggestions that helped improve the paper. 
We also thank Jos\'{e} Mario Mart\'{i}nez, Emeritus Professor of the 
University of Campinas, for fruitful discussions. 
The work of M.S. was supported by the international joint doctoral education fund of Beihang University.
The work of D.G. was supported in part by CNPq (Conselho Nacional de Desenvolvimento Cient\'ifico e Tecnol\'{o}gico) Grant 305213/2021-0 and by FAPESC (Funda\c{c}\~ao de Amparo a Pesquisa e Inova\c{c}\~ao do Estado de Santa Catarina) Grant 2024TR002238. 
W.J. and H.W. research is supported by the Natural Sciences and
Engineering Research Council of Canada. A.M. was partially supported by the ANR project EVARISTE (ANR-24-CE23-1621).
C.L. is supported in part by FAPESP (Grant 2023/08706-1) and CNPq (Grants 305227/2022-0 and 404616/2024-0).

\thispagestyle{empty}

\cleardoublepage
\label{ind:index}
\printindex

\addcontentsline{toc}{section}{Bibliography}

\end{document}